\documentclass[a4paper,reqno]{amsart}

\usepackage{graphicx}
\usepackage{mathptmx}
\usepackage[usenames,dvipsnames]{color}
\usepackage{microtype}
\usepackage{amsmath}
\usepackage{amssymb}
\usepackage{amsthm}
\usepackage{txfonts}
\usepackage{dsfont}
\usepackage{mathrsfs}
\usepackage{subfig}
\usepackage{enumitem}
\usepackage{ulem}

\theoremstyle{plain}
\newtheorem{theorem}{Theorem}[section]
\newtheorem{lemma}[theorem]{Lemma}
\newtheorem{proposition}[theorem]{Proposition}

\theoremstyle{definition}

\newtheorem{example}[theorem]{Example}
\newtheorem{remark}[theorem]{Remark}
\newtheorem{notation}[theorem]{Notation}

\allowdisplaybreaks

\newcommand{\e}{\mathrm{e}}

\begin{document}

\title{On the asymptotics of the $\alpha$-Farey transfer operator}

\thanks{The first two authors were supported by the German Research Foundation (DFG) grant 
\textit{Renewal Theory and Statistics of Rare Events in Infinite Ergodic Theory} ({Gesch{\"a}ftszeichen} KE 1440/2-1).}

\author{J.~Kautzsch}
\address{Fachbereich Mathematik, Universit\"at Bremen, 28359 Bremen, Germany}
\email{kautzsch@math.uni-bremen.de}

\author{M.~Kesseb\"ohmer}
\address{Fachbereich Mathematik, Universit\"at Bremen, 28359 Bremen, Germany}
\email{mhk@math.uni-bremen.de}

\author{T.~Samuel}
\address{Fachbereich Mathematik, Universit\"at Bremen, 28359 Bremen, Germany}
\email{tony@math.uni-bremen.de}

\author{B.~O.~Stratmann}
\address{Fachbereich Mathematik, Universit\"at Bremen, 28359 Bremen, Germany}
\email{bos@math.uni-bremen.de}

\maketitle

\begin{abstract}
We study the asymptotics of iterates of the transfer operator for non-uniformly hyperbolic $\alpha$-Farey maps.  We provide a family of observables which are Riemann integrable, locally constant and of bounded variation, and for which the iterates of the transfer operator, when applied to one of these observables, is not asymptotic to a constant times the wandering rate on the first element of the partition $\alpha$.  Subsequently, sufficient conditions on observables are given under which this expected asymptotic holds. In particular, we obtain an extension theorem which establishes that, if the asymptotic behaviour of iterates of the transfer operator is known on the first element of the partition $\alpha$, then the same asymptotic holds on any compact set bounded away from the indifferent fixed point.
\end{abstract}

%\keywords{
%Infinite Ergodic Theory \and Distributional Limit Laws \and $\alpha$-Farey Maps.}

%\subclass{37A40, 37A25, 37A50, 60K05}

\section{Introduction}\label{intro}

Expanding maps of the unit interval have been widely studied in the last decades and the associated transfer operators have proven to be of vital importance in solving problems concerning the statistical behaviour of the underlying interval maps \cite{B:2000,M:1991}.

In recent years an increasing amount of  interest has developed in maps which are expanding everywhere except on an unstable fixed point (that is, an indifferent fixed point)  at which trajectories are considerably slowed down.  This leads to an interplay of chaotic and regular dynamics, a characteristic of intermittent systems \cite{PM:1980,S:1988}.  From an ergodic theory viewpoint, this phenomenon leads to an absolutely continuous invariant measure having infinite mass.  Therefore, standard methods of ergodic theory cannot be applied in this setting; indeed it is wellknown that Birkhoff's ergodic theorem does not hold under these circumstances, see for instance \cite{JA:1997}.

In this paper we will be concerned with $\alpha$-Farey maps, see Figure~\ref{Fig1}.  These maps are of great interest since they are piecewise linear and expanding everywhere except for  at the indifferent fixed point where they have (right) derivative one. This makes the $\alpha$-Farey maps a simple model for studying the physical phenomenon of intermittency \cite{PM:1980}.  Moreover, an induced version of the $\alpha$-Farey maps are given by the $\alpha$-L\"uroth maps introduced in \cite{Lu:1883}, which have significant meaning in number theory, see for instance \cite{BBDK:1996,KMS:2012}.  

Thaler \cite{T:2000} was the first to discern the asymptotics of the transfer operator of a class of interval maps preserving an infinite measure.  This class of maps, to which the $\alpha$-Farey maps do not belong, have become to be known as Thaler maps.  In an effort to generalise this work, by combining  renewal theoretical arguments and functional analytic techniques, a new approach to estimate the decay of correlation of a dynamical system was achieved by Sarig \cite{S:2002}. Subsequently, Gou\"{e}zel \cite{G:2004,G:2005,G:2011} generalised these methods.  Using these ideas and employing the methods of Garsia and Lamperti \cite{GL:1969}, Erickson \cite{E:1970} and Doney \cite{D:1997}, recently Melbourne and Terhesiu \cite[Theorem 2.1 to 2.3]{MT:2011} proved a landmark result on the asymptotic rate of convergence of iterates of the induced transfer operator and showed that these result can be applied to Gibbs-Markov maps, Thaler maps, AFN maps, and Pomeau-Manneville maps.  Thus, the question which naturally arises is, whether this asymptotic rate can be related to the asymptotic rate of convergence of iterates of the actual transfer operator.  The results of this paper give some positive answers to this question for $\delta$-expansive $\alpha$-Farey maps.

As mentioned above, in this paper, we will consider the $\alpha$-Farey map $F_{\alpha} \colon [0, 1] \to [0, 1]$, which is given for a countable infinite partition $\alpha \coloneqq \{ A_{n} : n \in \mathbb{N}\}$ of $(0, 1)$ by non-empty intervals $A_{n}$.  It is assumed throughout that the atoms of $\alpha$ are ordered from right to left, starting with $A_{1}$, and that these atoms only accumulate at zero.  Further, we assume that $A_{n}$ is right-open and left-closed, for all natural numbers $n$.  We define the $\alpha$-\textit{Farey map}  $F_{\alpha} \colon [0, 1] \to [0, 1]$ by
\[
F_{\alpha}(x) \coloneqq \begin{cases}
(1-x)/a_{1} & \text{if} \; x \in \overline{A}_{1} \coloneqq A_{1} \cup \{ 1 \},\\
a_{n-1}(x-t_{n+1})/a_{n}+t_{n} & \text{if} \; x \in A_{n}, \; \text{for} \; n \geq 2,\\
0 & \text{if} \; x = 0,
\end{cases}
\]
where $a_{n}$ is equal to the Lebesgue measure $\lambda(A_{n})$ of the atom $A_{n} \in \alpha$ and $t_{n} \coloneqq \sum_{k = n}^{\infty} a_{k}$ denotes the Lebesgue measure of the $n$-th tail $\bigcup_{k = n}^{\infty} A_{k}$ of $\alpha$, see Figure~\ref{Fig1}.  Throughout, we will assume that the partition $\alpha$ satisfies the condition that the sequence $( t_{n} )_{n \in \mathbb{N}}$ is not summable.  For $\delta \in (0, 1]$, an $\alpha$-Farey map $F_{\alpha}$ is said to be \textit{$\delta$-expansive} if the sequence $(a_{n})_{n\in\mathbb{N}}$ is regularly varying of order $- (1 + \delta)$, that is, if there exists a slowly varying function $l \colon \mathbb{R} \to \mathbb{R}$ such that $a_{n} = \delta l(n) n^{-(1+\delta)}$, for all $n \in \mathbb{N}$.  (Recall that $l \colon [a, \infty) \to \mathbb{R}$ is called a \textit{slowly varying function}, if it is measurable, locally Riemann integrable and $\lim_{x \to \infty} l(\eta x)/l(x) = 1$, for each $\eta > 0$ and for some $a \in \mathbb{R}$, see \cite{BGT:1987,Seneta:1976} for  further details.)  In this situation, \cite[Theorem 1.5.10]{BGT:1987} implies that
\[
\lim_{n \to \infty} \frac{l(n) n^{-\delta}}{t_{n}}
= \lim_{n \to \infty} \frac{l(n) n^{-\delta}}{\sum_{k = n}^{\infty} a_{k}}
= \lim_{n \to \infty} \frac{l(n) n^{-\delta}}{\sum_{k = n}^{\infty} \delta l(n) n^{-(1+\delta)}}
= 1.
\]
Therefore, the Lebesgue measure of the $n$-th tail of $\alpha$ is asymptotic to a regularly varying function of order $-\delta$.  Thus, $\delta$-expansive implies expansive of order $\delta$ in the sense of \cite{KMS:2012}.  However, an expansive $\alpha$-Farey map of order $\delta$ is not necessarily $\delta$-expansive.

 \begin{figure}[htbp]
 \begin{center}
 \subfloat[$\delta = 65/128$.]{
 \scalebox{0.5}{
 \includegraphics{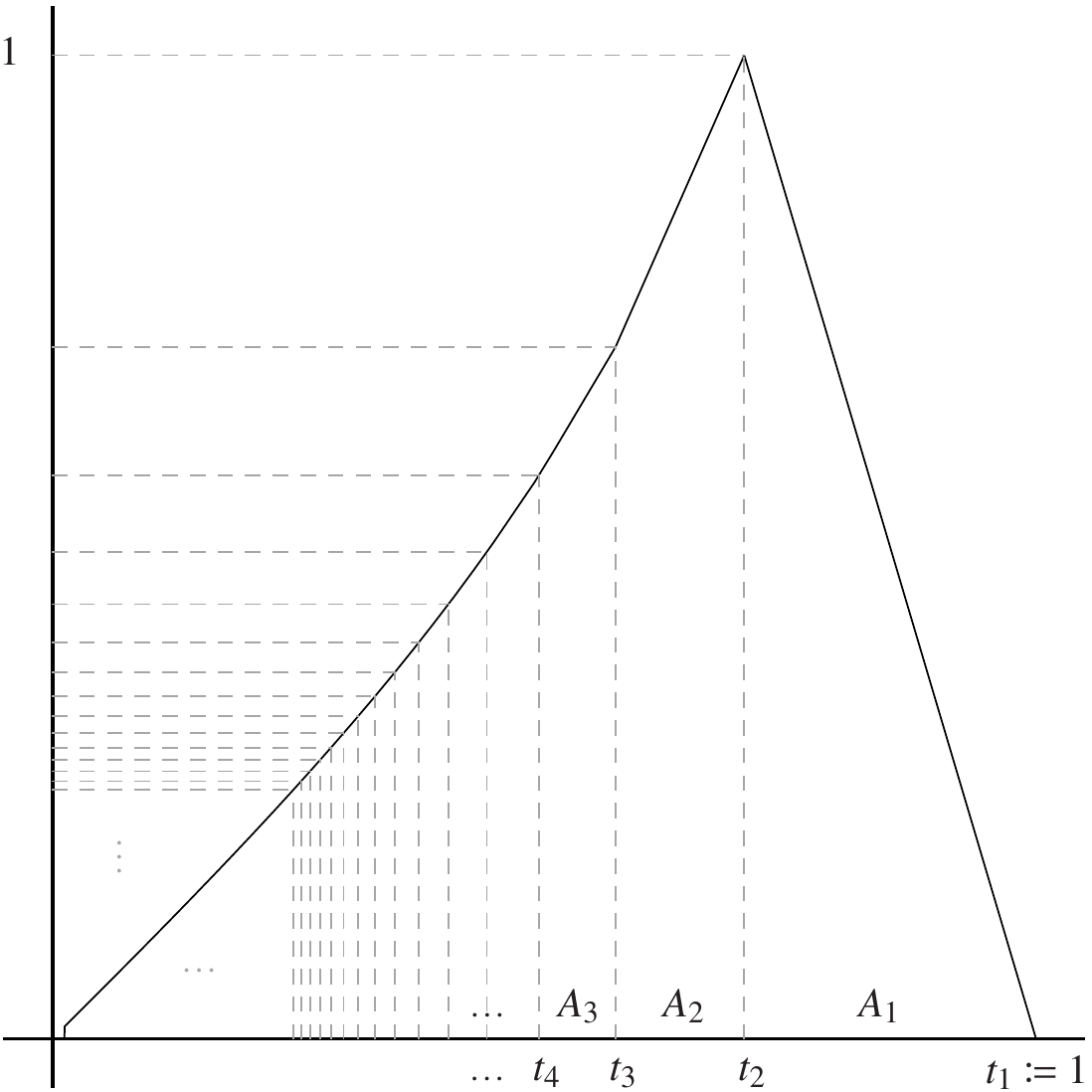}
 }}
 \hspace{0.5em}
 \subfloat[$\delta = 1$.]{
 \scalebox{0.5}{
 \includegraphics{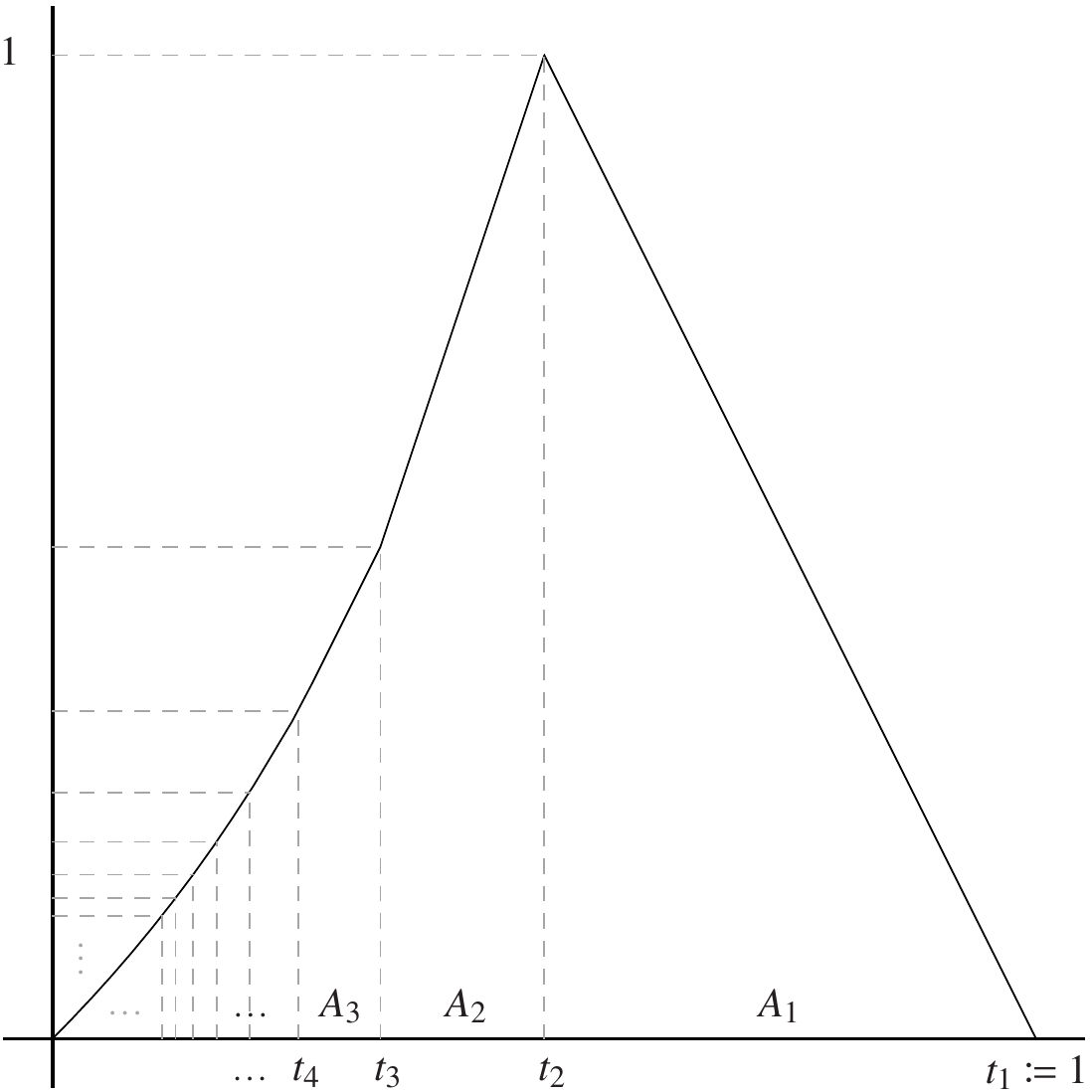}
 }}
 \end{center}
 \caption{The $\alpha$-Farey map, where $t_{n} = n^{-\delta}$, for all $n \in \mathbb{N}$.}
 \label{Fig1}
 \end{figure}

Throughout, let $\mu_{\alpha}$ denote the $F_{\alpha}$-invariant measure which is determined by
\begin{equation}\label{eq:RND}
h_{\alpha} \coloneqq \frac{\mathrm{d}\mu_{\alpha}}{\mathrm{d}\lambda} = \sum_{n \in \mathbb{N}} \frac{t_{n}}{a_{n}} \mathds{1}_{A_{n}}
\end{equation}
and let $\mathscr{B}$ denote the Borel $\sigma$-algebra of $[0, 1]$.  Here and in the sequel, for a given Borel set $B \in \mathscr{B}$, we let $\mathds{1}_{B}$ denote the indicator function on $B$.  It is verified in \cite{KMS:2012} that, since the sequence $(a_{n})_{n\in\mathbb{N}}$ is regularly varying of order $-(1+\delta)$, the map $F_{\alpha}$ is conservative, ergodic and measure preserving on the infinite and $\sigma$-finite measure space $([0, 1], \mathscr{B}, \mu_{\alpha})$.  The dynamical system $( [0, 1], \mathscr{B}, \mu_{\alpha}, F_{\alpha})$ will be referred to as a \textit{$\alpha$-Farey system}.  

Following the definitions and notations of \cite{Billingsley:1995}, throughout, we let $\mathcal{L}_{\mu_{\alpha}}^{1}([0, 1])$ (respectively $\mathcal{L}_{\lambda}^{1}([0, 1])$) denote the class of measurable functions $f$ with domain $[0, 1]$ for which  $\lvert f \rvert$ is $\mu_{\alpha}$-integrable (respectively $\lambda$-integrable), and for $f \in \mathcal{L}_{\mu_{\alpha}}^{1}([0, 1])$ (respectively $\mathcal{L}_{\lambda}^{1}([0, 1])$), define $\lVert f \rVert_{\mu_{\alpha}}$ (respectively $\lVert f \rVert_{\lambda}$) by
\[
\lVert f \rVert_{\mu_{\alpha}} \coloneqq \int \lvert f \rvert \; \mathrm{d}\mu_{\alpha} \quad  \left(\text{respectively} \; \lVert f \rVert_{\lambda} \coloneqq \int \lvert f \rvert \; \mathrm{d}\lambda \right).
\]
Further, for a given measurable function $w$, we set $\lVert w \rVert_{\infty} \coloneqq \sup_{x \in [0, 1]} \rvert w(x) \lvert$.

 \begin{figure}[htbp]
 \begin{center}
 \subfloat[$\delta = 65/128$.]{
 \scalebox{0.48}{
 \includegraphics{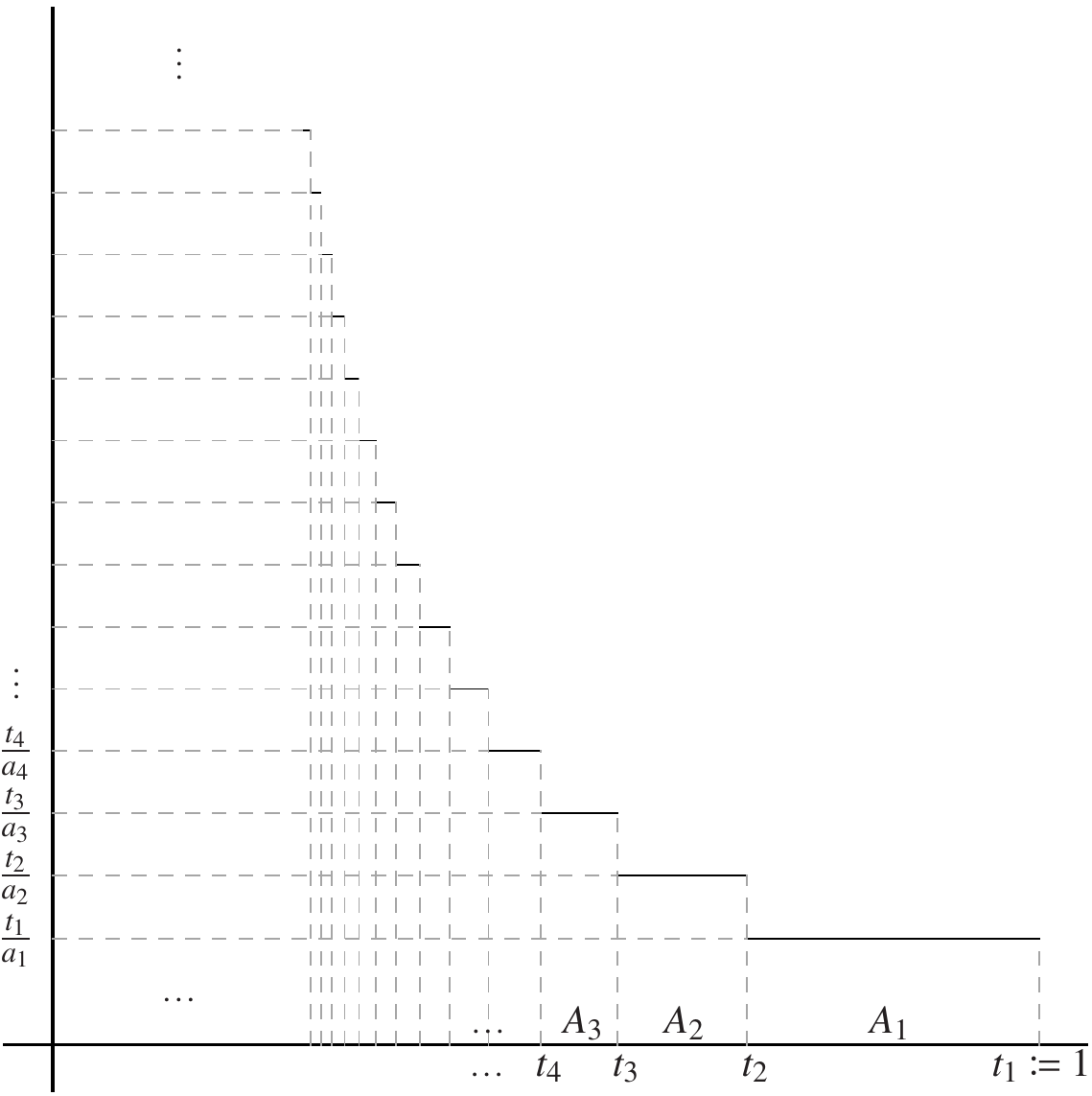}
 }}
 \hspace{0.5em}
 \subfloat[$\delta = 1$.]{
 \scalebox{0.48}{
  \includegraphics{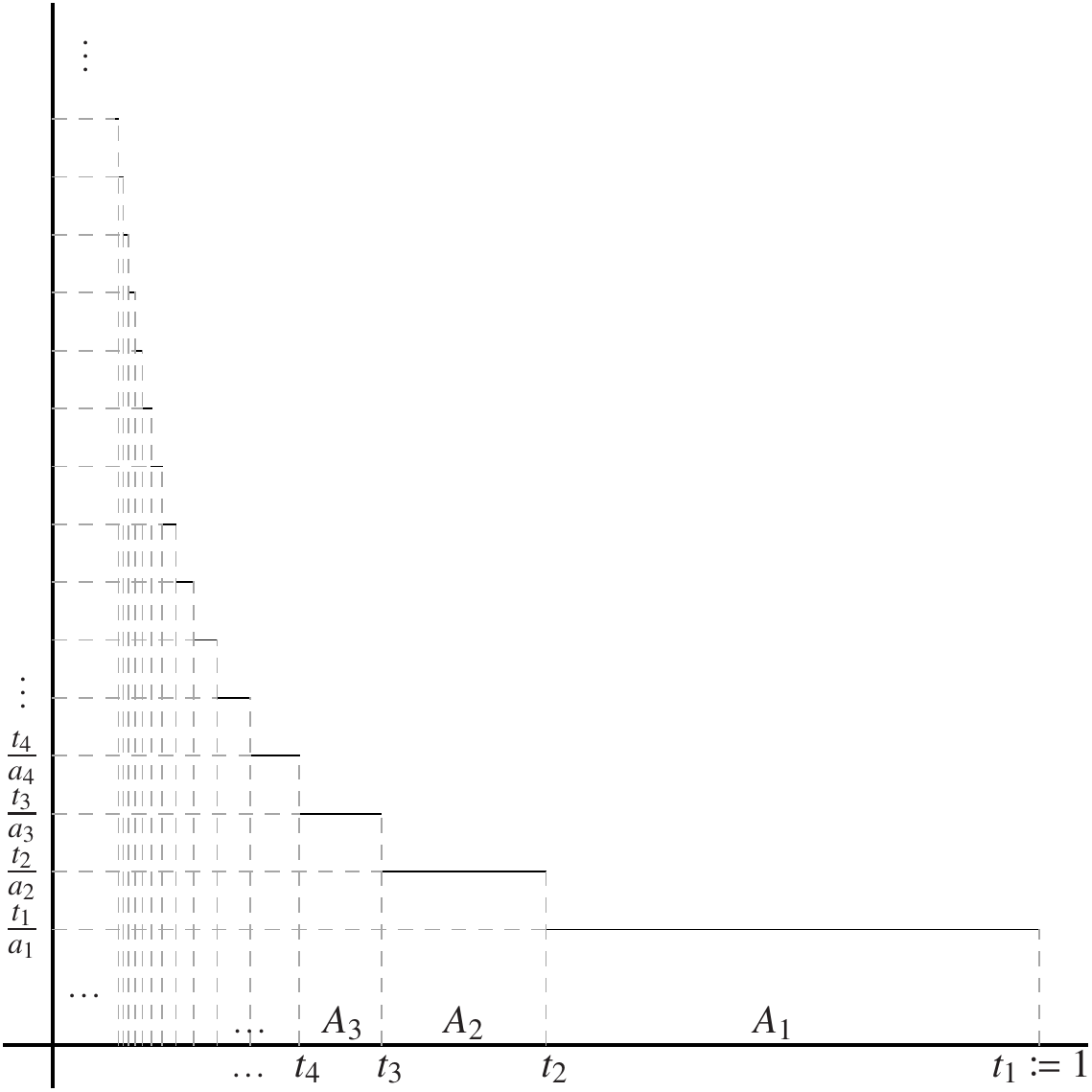}
 }}
 \end{center}
\caption{Plot of the density function $h_{\alpha}$ for the $\alpha$-Farey map, where $t_{n} = n^{-\delta}$ for all $n \in \mathbb{N}$.}
 \label{Fig4}
 \end{figure}

The \textit{$\alpha$-Farey transfer operator} $\widehat{F}_{\alpha} \colon \mathcal{L}^{1}_{\mu_{\alpha}}([0, 1]) \to \mathcal{L}^{1}_{\mu_{\alpha}}([0,1])$ is the positive linear operator given by
\begin{equation}\label{eq:defofThat}
\widehat{F}_{\alpha}(v) \coloneqq \sum_{n \in \mathbb{N}} \left( \frac{t_{n+1}}{t_{n}} v \circ F_{\alpha, 0} + \left( 1 - \frac{t_{n+1}}{t_{n}}\right) v \circ F_{\alpha, 1} \right) \cdot \mathds{1}_{A_{n}},
\end{equation}
where $F_{\alpha, 0} \coloneqq (F_{\alpha}\lvert_{[0,t_{2}]})^{-1}$ and $F_{\alpha, 1} \coloneqq (F_{\alpha}\lvert_{[t_{2}, 1]})^{-1}$ refer to the inverse branches of $F_{\alpha}$.  In particular, for all $v \in \mathcal{L}^{1}_{\mu_{\alpha}}([0,1])$ and all measurable functions $w$ with $\lVert w \rVert_{\infty} < \infty$,
\begin{equation}\label{eq:transfer_cheat}
\int \widehat{F}_{\alpha}(v) \cdot w \,\mathrm{d}\mu_{\alpha} = \int v \cdot w \circ F_{\alpha} \,\mathrm{d}\mu_{\alpha}.
\end{equation}
(The above equality is a direct consequence of \cite[Lemma 2.5]{KMS:2012}.)  Note that the equality given in \eqref{eq:transfer_cheat} is the usual defining relation for the \textit{transfer operator} of $F_{\alpha}$.  However, the relation in \eqref{eq:transfer_cheat} only determines values of the transfer operator of $F_{\alpha}$ applied to an observable $\mu_{\alpha}$-almost everywhere.  Thus the $\alpha$-Farey transfer operator is a version of the transfer operator of $F_{\alpha}$.

In order to state our main theorems, we will also require the following function spaces.  We let $\phi \colon \overline{A}_{1} \to \mathbb{N} \cup \{ +\infty \}$ denote the \textit{first return time}, given by $\phi(y) \coloneqq \inf \{ n \in \mathbb{N} : F_{\alpha}^{n}(y) \in \overline{A}_{1} \}$, and we write $\{\phi = n \} \coloneqq \{ y \in A_{1} : \phi(y) = n\}$.  Let $\beta_{\alpha}$ denote the countable-infinite partition $\{ \{ \phi = n \} : n \in \mathbb{N} \}$ of $A_{1}$ and let $\mathcal{B}_{\alpha}$ denote the set of functions with domain $[0, 1]$ that are supported on a subset of $\overline{A}_{1}$ and which have finite $\lVert \cdot \rVert_{\mathcal{B}_{\alpha}}$-norm, where $\lVert \cdot \rVert_{\mathcal{B}_{\alpha}} \coloneqq \lVert \cdot \rVert_{\infty} + D_{\alpha}(\cdot)$ and where
\[
D_{\alpha}(f) \coloneqq \sup_{a \in \beta_{\alpha}}  \sup_{x \neq y \in a} \frac{\lvert f(x) - f(y) \rvert}{\lvert x - y \rvert}.
\]
In particular, if $f \in \mathcal{B}_{\alpha}$, then $f$ is Lipschitz continuous on each atom of $\beta_{\alpha}$, zero outside of $\overline{A}_{1}$ and bounded (everywhere).  We then define
\[
\mathcal{A}_{\alpha} \coloneqq \left\{ v \in \mathcal{L}_{\mu_{\alpha}}^{1}([0, 1]) : \lVert v \rVert_{\infty} < \infty\; \text{and} \; \widehat{F}^{n-1}_{\alpha}(v \cdot \mathds{1}_{A_{n}}) \in \mathcal{B}_{\alpha} \; \text{for all} \; n \in \mathbb{N}   \right\}.
\]
For examples of observables belonging to $\mathcal{A}_{\alpha}$, we refer the reader to Example~\ref{ex:ex1} and the discussion succeeding our main results, Theorems~\ref{thm:main2} and \ref{thm:main1}.  Let us also recall from \cite{KMS:2012} that the \textit{wandering rate} of $F_{\alpha}$ is given by
\[
w_{n} = w_{n}(F_{\alpha}) \coloneqq \mu_{\alpha} \left( \bigcup_{k=0}^{n-1} F_{\alpha}^{-k}(A_{1})\right) = \mu_{\alpha} \left( \bigcup_{k=1}^{n} A_{k} \right) = \sum_{k = 1}^{n} t_{k}.
\]
Further, as we will see in \eqref{eq:Karamata}, if $\delta \in (0, 1)$ and if the given $\alpha$-Farey system is $\delta$-expansive, then the wandering rate is regularly varying of order $1 - \delta$.  Also, in the case that $\delta = 1$, if
\[
\left( w_{n} / w_{\lceil n \cdot {w_{n}}^{-2} \rceil }\right)_{n\in\mathbb{N}}
\]
is a bounded sequence, then we say that the wandering rate $w_{n}$ is \textit{moderately increasing}.  Here and in the sequel for $r \in \mathbb{R}$ we let $\lceil r \rceil$ denote the smallest integer greater than or equal to $r$.

With the above preparations, we are now in a position to state the main results, Theorems~\ref{thm:main2} and \ref{thm:main1}. Theorem~\ref{thm:main2} provides mild conditions under which the asymptotic behavior of the iterates of an $\alpha$-Farey transfer operator \textit{restricted to} $A_{1}$ can be extended to all of $(0, 1]$ and is used in our proof of Theorem~\ref{thm:main1}.  (Note that, by \eqref{eq:Karamata}, any $\delta$-expansive $\alpha$-Farey system satisfies the requirements of Theorem~\ref{thm:main2}.)  One of the facets of Theorem~\ref{thm:main1} is that it gives sufficient conditions on observables which guarantee that iterates of an $\alpha$-Farey transfer operator applied to such an observable is asymptotic to a constant times the wandering rate.  These results complement \cite[Theorem 10.5]{MT:2011} and show that additional assumptions are required in \cite[Theorem 10.4]{MT:2011}.  Namely, in the case that $\delta = 1$, we show that the statement of \cite[Theorem 10.4]{MT:2011} holds true, with the additional assumption that the wandering rate is moderately increasing (Theorem~\ref{thm:main1}\ref{i}); for $\delta \in (1/2, 1)$, we provide an example which demonstrates that additional requirements are necessary for the expected convergence (Theorem~\ref{thm:main1}\ref{iii}) and provide sufficient conditions (Theorem~\ref{thm:main1}\ref{ii}).

\begin{theorem}
\label{thm:main2}
For an  $\alpha$-Farey system $([0, 1], \mathscr{B}, \mu_{\alpha}, F_{\alpha})$ for which the wandering rate satisfies the condition $\lim_{n \to \infty} w_{n} /w_{n+1} = 1$, we have that, if $v \in \mathcal{L}_{\mu_{\alpha}}^{1}([0, 1])$ satisfies
\[
\lim_{n \to +\infty} w_{n} \widehat{F}_{\alpha}^{n} (v) = \Gamma_{\delta} \int v \, \mathrm{d}\mu_{\alpha}
\]
uniformly on $\overline{A}_{1}$,  then the same holds on any compact subset of $(0, 1]$. The same statement holds when replacing uniform convergence by almost everywhere uniform convergence.
\end{theorem}

\begin{remark}
For $\delta \in (0, 1)$, a $\delta$-expansive $\alpha$-Farey system has wandering rate satisfying $\lim_{n \to \infty} w_{n} /w_{n+1} = 1$.
\end{remark}

\begin{theorem}\label{thm:main1}
Let $([0, 1], \mathscr{B}, \mu_{\alpha}, F_{\alpha})$ be a $\delta$-expansive $\alpha$-Farey system.
\begin{enumerate}[label={(\roman*)},leftmargin=*]
\item\label{i} Let $\delta\!=\!1$ and assume that the wandering rate is moderately increasing.\,If\,$v\!\in\!\mathcal{A}_{\alpha}$\,and\,if
\begin{equation}\label{eq:main1}
\sum_{k = 1}^{\infty} \left\lVert \widehat{F}_{\alpha}^{k-1}(v \cdot \mathds{1}_{A_{k}} ) \right\rVert_{\infty} < + \infty,
\end{equation}
then uniformly on compact subsets of $(0, 1]$,
\begin{equation}\label{lim:parti}
\lim_{n \to \infty} w_{n} \widehat{F}_{\alpha}^{n}(v) = \int v \, \mathrm{d}\mu_{\alpha}.
\end{equation}
\item\label{ii} For $\delta \in (1/2,1]$, if $v \in \mathcal{L}_{\lambda}^{1}([0,1])$ with $\lvert v(1) \rvert$ bounded and if 
\begin{enumerate}[leftmargin=*]
\item the sequence $( D_{\alpha}(\mathds{1}_{A_{1}} \cdot \widehat{F}_{\alpha}^{n-1} (v \cdot \mathds{1}_{A_{n}})) )_{n \in \mathbb{N}}$ is bounded and
\item there exist constants $c > \lvert v(1) \rvert$ and $\eta \in (0, \delta)$ with $\lVert v \cdot \mathds{1}_{A_{n}} \rVert_{\infty} \leq c n^{\eta}$, for all $n \in \mathbb{N}$.
\end{enumerate}
then uniformly on compact subsets of $(0, 1]$,
\begin{equation}\label{lim:partii}
\lim_{n \to \infty} w_{n}  \widehat{F}_{\alpha}^{n}(v/h_{\alpha})  = \Gamma_{\delta} \int v/h_{\alpha} \, \mathrm{d}\mu_{\alpha}.
\end{equation}
Here, $\Gamma_{\delta} \coloneqq (\Gamma(1 + \delta) \Gamma(2-\delta))^{-1}$ and $\Gamma$ denotes the Gamma function.
\item\label{iii} For $\delta \in (1/2, 1)$, there exists a positive, locally constant, Riemann integrable function $v \in \mathcal{A}_{\alpha}$ of bounded variation satisfying the inequality in \eqref{eq:main1}, such that, for all $x \in \overline{A}_{1}$, 
\begin{equation}\label{eq:counterexample}
\liminf_{n \to \infty} w_{n}  \widehat{F}_{\alpha}^{n}(v)(x) =  \Gamma_{\delta} \int v \, \mathrm{d} \mu_{\alpha}
\quad \text{and} \quad
\limsup_{n \to \infty} w_{n}  \widehat{F}_{\alpha}^{n}(v)(x) = +\infty.
\end{equation}
\end{enumerate}
\end{theorem}

\begin{remark}
It is immediate that if $a_{n} = n^{-1}(n+1)^{-1}$, then $t_{n} = n^{-1}$ and $w_{n} \sim \ln(n)$, and that these parameters give rise to an example of an $\alpha$-Farey system which satisfies the conditions of Theorem~\ref{thm:main1}\ref{i}.  Indeed there exist many examples of $\alpha$-Farey systems for which the conditions of Theorem~\ref{thm:main1}\ref{i} are satisfied, but where the wandering rate behaves very differently to the function $n \mapsto \ln(n)$.   Letting $\delta = 1$, as we will see in Lemma~\ref{lem:powerslowly}\ref{SV(vi)}, the sequence $(w_{n})_{n \in \mathbb{N}}$ is slowly varying and $\lim_{n \to \infty} n t_{n} /w_{n} = 0$.  We also have that
\[
w_{n} = \sum_{j = 1}^{n} t_{j} = n \sum_{j = n + 1}^{\infty} a_{j} + \sum_{j = 1}^{n} j a_{j} = n t_{n + 1} + \sum_{j = 1}^{n} j a_{j}.
\]
Using this we deduce the following.
\begin{enumerate}
\item If $\displaystyle{a_{n} = n^{-2} (\ln(n))^{-1/2}} \e^{(\ln(n))^{1/2}}$, then $t_{n} \sim n^{-1} (\ln(n))^{-1/2} \e^{(\ln(n))^{1/2}}$ and $\displaystyle{w_{n} \sim \e^{(\ln(n))^{1/2}}}$.\\
\item If $\displaystyle{a_{n - 16} = n^{-2} \kappa(n) \e^{\ln(n)/\ln(\ln(n))}}$, then $\displaystyle{t_{n} \sim n^{-1} \kappa(n) \e^{\ln(n)/\ln(\ln(n))}}$ and $\displaystyle{w_{n} \sim \e^{\ln(n)/\ln(\ln(n))}}$,\\
where $\kappa(n) = (\ln(\ln(n)) - 1) (\ln(\ln(n)))^{-2}$.
\end{enumerate}
Indeed the above two sets of parameters give rise to examples of $1$-expansive $\alpha$-Farey systems whose wandering rate is moderately increasing.  Moreover,
\[
\lim_{n \to \infty} \frac{\ln(n)}{\e^{(\ln(n))^{1/2}}} = 0,
\quad
\lim_{n \to \infty} \frac{\ln(n)}{\e^{\ln(n)/\ln(\ln(n))}} = 0
\quad \text{and} \quad
\lim_{n \to \infty} \frac{\e^{(\ln(n))^{1/2}}}{\e^{\ln(n)/\ln(\ln(n))}} = 0,
\]
demonstrating that two moderately increasing wandering rates, although they are all slowly varying, do not have to be asymptotic to each other nor to the function $n \mapsto \ln(n)$.
\end{remark}

\begin{remark}
In the case that $F_{\alpha}$ is a $1$-expansive $\alpha$-Farey map, we have that the wandering rate $w_{n}$ is a slowly varying function.  We remark here that it is not the case that every slowly varying function is moderately increasing, namely, it is not the case that if $l \colon [0, \infty] \to \mathbb{R}$ is a slowly varying function, then the sequence
\begin{align}\label{eq:mod_varying_exam}
\left( l(n) / l(\lceil n \cdot {l(n)}^{-2} \rceil ) \right)_{n\in\mathbb{N}}
\end{align}
is bounded.  For instance consider the following.  Let $( c_{k} )_{k \in \mathbb{N}}$ be a decreasing sequence of positive real numbers which converge to zero and, for $k \in \mathbb{N}$, set
\[
x_{k + 1} = \frac{k}{2 {c_{k}}^{2}} \quad \text{and} \quad b_{k+1} = (c_{k} - c_{k+1}) x_{k+1} + b_{k},
\]
where $x_{1} = b_{1} = 0$.  We define $m \colon [0, \infty) \to \mathbb{R}$ by
\[
m(x) \coloneqq c_{k} x + b_{k},
\]
for $x \in [x_{k}, x_{k+1}]$.  The function $l \colon [1, \infty] \to \mathbb{R}$ defined by $l(x) \coloneqq \e^{m(\ln(x))}$ is, by construction, slowly varying.  However, the sequence given in \eqref{eq:mod_varying_exam} is unbounded.  (We are grateful to Fredrik Ekstr\"om for providing this example).
\end{remark}

\begin{remark}
If in the definition of the norm $\lVert \cdot \rVert_{\mathcal{B}_{\alpha}}$, one replaces the norm $\lVert \cdot \rVert_{\infty}$ by the \textit{essential supremum norm} $\lVert \cdot \rVert_{\textup{ess sup}}$, then by appropriately adapting the proofs given in the sequel, one can obtain a proof of Theorem~\ref{thm:main1} where the uniform convergence on compact subsets of $(0, 1]$ is replaced by uniform convergence almost everywhere on compact subsets of $(0, 1)$.
\end{remark}

\begin{remark}
The first part of the proof of Theorem~\ref{thm:main1}~\ref{i} and \ref{ii} are inspired by the first paragraph in the proof of \cite[Theorem 10.4]{MT:2011}.  
\end{remark}

The structure of this paper is as follows.  In Section~\ref{sec:pre} we collect basic properties of $\alpha$-Farey maps and their corresponding transfer operators.  In Section~\ref{J_Diploma_Thesis} we provide a proof of Theorem~\ref{thm:main2}.  This proof is inspired by arguments originally presented in \cite{JK:2011}.  Then in Section~\ref{sec:main1} we present the proof of Theorem~\ref{thm:main1}, breaking the proof into three constituent parts.  In Section~\ref{section:Counterexamples2.2} we obtain part \ref{i} and give explicit examples of observables satisfying the given properties.  In Section~\ref{MRE_alpha_delta_in_051} we prove part \ref{ii}, for explicit examples of observables which satisfy the pre-requests of Theorem~\ref{thm:main1}~\ref{ii} we refer the reader to Remark~\ref{rmk:rmk2}.  Finally we conclude with Section~\ref{section:Counterexamples} where part \ref{iii} is proven using a constructive argument.

Before we conclude this section with a series of remarks, Remarks~\ref{rmk:rmk1} to \ref{rmk:rmk3}, in which we comment on how Theorem~\ref{thm:main1}, and hence Theorem \ref{thm:main2}, complement the results obtained in \cite{KS:2008,T:2000}, we introduce the \textit{Perron-Frobenius operator} $\mathcal{P}_{\alpha} \colon \mathcal{L}_{\lambda}^{1}([0, 1]) \to \mathcal{L}_{\lambda}^{1}([0, 1])$ which is defined by 
\[
\mathcal{P}_{\alpha} (f) (x) \coloneqq \sum_{y\in F_{\alpha}^{-1}(x)} \lvert F_{\alpha}'(y) \rvert^{-1} f(y),
\]
where $F_{\alpha}'$ denotes the right derivative of $F_{\alpha}$ and where $F_{\alpha}'(1) \coloneqq -{a_{1}}^{-1}$.  (Note, by construction, if $F_{\alpha}$ is $\delta$-expansive, then the right derivative of $F_{\alpha}$ at zero is equal to one.)  A useful relation between the operators $\mathcal{P}_{\alpha}$ and $\widehat{F}_{\alpha}$ is that
\begin{equation}\label{eq:PF-T}
\widehat{F}_{\alpha}(f) = \mathcal{P}_{\alpha}(h_{\alpha} \cdot f) / h_{\alpha}.
\end{equation}
We refer the reader to \cite[p. 1001]{KMS:2012} for a proof of the equality in \eqref{eq:PF-T}.

\begin{remark}\label{rmk:rmk1}
For certain interval maps $T \colon [0, 1] \to [0,1]$ with two monotonically increasing, differentiable branches whose invariant measure has infinite mass and whose tail probabilities are regularly varying with exponent $-\delta \in [-1,0)$, Thaler \cite{T:2000} discerned the precise asymptotic behaviour of iterates of the associated Perron-Frobenius operator $\mathcal{P}$, namely, that for all Riemann integrable functions $u$ with domain $[0, 1]$, one has that
\begin{equation}\label{eq:Thaler:KS2008}
\lim_{n \to +\infty} w_{n}(T) \mathcal{P}^{n}(u) = \Gamma_{\delta} \left( \int u \, \mathrm{d}\lambda \right)  h
\end{equation}
uniformly almost everywhere on compact subsets of $(0,1]$.  Here, $h$ denotes the associated invariant density and $w_{n}(T)$ denotes the wandering rate of $T$.  However, $\alpha$-Farey maps do not fall into this class of interval maps.  Using the relationship between the transfer and the Perron-Frobenius operator, Theorem~\ref{thm:main1}~\ref{ii} together with the assumption that the Banach space of functions of bounded variation\label{Page:BV-Banach} with the norm $\lVert \cdot \rVert_{\textup{ess sup}} + \mathrm{Var}(\cdot)$ satisfies certain functional analytic conditions (namely, conditions (H1) and (H2) given in Section~\ref{sec:pre}), show that Thaler's result can be extended to $\delta$-expansive $\alpha$-Farey maps.  Results of this form have also been obtained in \cite{ZT:2006} for AFN maps.  (Note, an $\alpha$-Farey map is also not an AFN map.)
\end{remark}

\begin{remark}\label{rmk:rmk2}
{Kesseb\"ohmer} and Slassi \cite{KS:2008} showed that for the classical Farey map the convergence given in \eqref{eq:Thaler:KS2008} holds uniformly almost everywhere on $[1/2, 1]$ for convex \mbox{$C^{2}$-observables}.  Likewise, for a $\delta$-expansive $\alpha$-Farey map, Theorems~\ref{thm:main1}~\ref{ii} implies that if $u$ is a convex {$C^{2}$-observable}, then the convergence in \eqref{eq:Thaler:KS2008} holds uniformly on compact subsets of $(0,1]$.  To see that a convex $C^{2}$-observable satisfies the requirements of Theorem~\ref{thm:main1}~\ref{ii}, one employs arguments similar to those used in Example~\ref{ex:ex1} together with \eqref{eq:RND} and \eqref{eq:Thatonpetals1}.
 \end{remark}
 
\begin{remark}\label{rmk:rmk3}
The consequences of Theorem~\ref{thm:main1} go even further, in that for a \mbox{$\delta$-expansive} \mbox{$\alpha$-Farey} map, we are able to obtain that the convergence given in \eqref{eq:Thaler:KS2008} holds uniformly on compact subsets of $(0, 1]$, for certain non-Riemann integrable observables which are not necessarily bounded.  For instance, if $v$ is an observable such that $v \cdot \mathds{1}_{A_{n}} = n^{\eta} \mathds{1}_{A_{n}}$, $v(0) = 0$ and $v(1) = 1$, for some $\eta \in (0, \delta)$, then, as we will see in Lemma~\ref{lem:partIIclaim3}, since $0 \leq F^{n-1}(\mathds{1}_{A_{n}}) \leq t_{n}\mathds{1}_{\overline{A}_{1}}$, this observable fulfils the conditions of Theorem~\ref{thm:main1}~\ref{ii} and it is neither Riemann integrable nor is it bounded.
\end{remark}

\begin{notation}
We use the symbol $\sim$ between the elements of two sequences of real numbers $(b_{n})_{n \in \mathbb{N}}$ and $(c_{n})_{n \in \mathbb{N}}$ to mean that the sequences are asymptotically equivalent, namely that $\lim_{n \to +\infty} b_{n}/c_{n} = 1$.  We use the Landau notation $b_{n} = \mathfrak{o}(c_{n})$, if $\lim_{n \to +\infty} b_{n}/c_{n} = 0$.  The same notation is used between two real-valued function $f$ and $g$, defined on the set of real numbers $\mathbb{R}$, positive real numbers $\mathbb{R}^{+}$, natural numbers $\mathbb{N}$ or non-negative integers $\mathbb{N}_{0}$.  Specifically, if $\lim_{x \to +\infty} f(x)/g(x) = 1$, then we will write $f \sim g$, and if $\lim_{x \to +\infty} f(x)/g(x) = 0$, then we will write $f \in \mathfrak{o}(g)$.
\end{notation}

\section{The $\alpha$-Farey system}\label{sec:pre}

The map $G_{\alpha} \colon \overline{A}_{1} \to \overline{A}_{1}$ defined by
\[
G_{\alpha}(x) \coloneqq \begin{cases}
F_{\alpha}^{\phi(x)}(x) & \text{if} \; x \in A_{1},\\
t_{2} & \text{if} \; x = 1,
\end{cases}
\]
is called the \textit{first return map} and it is well known that $G_{\alpha}$ is conservative, ergodic and measure preserving on $(\overline{A}_{1}, \mathscr{B}\lvert_{\overline{A}_{1}}, \mu_{\alpha}\lvert_{\overline{A}_{1}})$, see for instance \cite[Propositions 1.4.8 and 1.5.3]{JA:1997}.  From this point on, we write $\mu_{\alpha}$ for both $\mu_{\alpha}$ and $\mu_{\alpha}\lvert_{\overline{A}_{1}}$ and $\mathscr{B}$ for both $\mathscr{B}$ and $\mathscr{B}\lvert_{\overline{A}_{1}}$.  Also, throughout, unless otherwise stated, we assume that $F_{\alpha}$ is $\delta$-expansive.

 \begin{figure}[htbp]
 \begin{center}
 \subfloat[$\delta = 65/128$.]{
 \scalebox{0.5}{
 \includegraphics{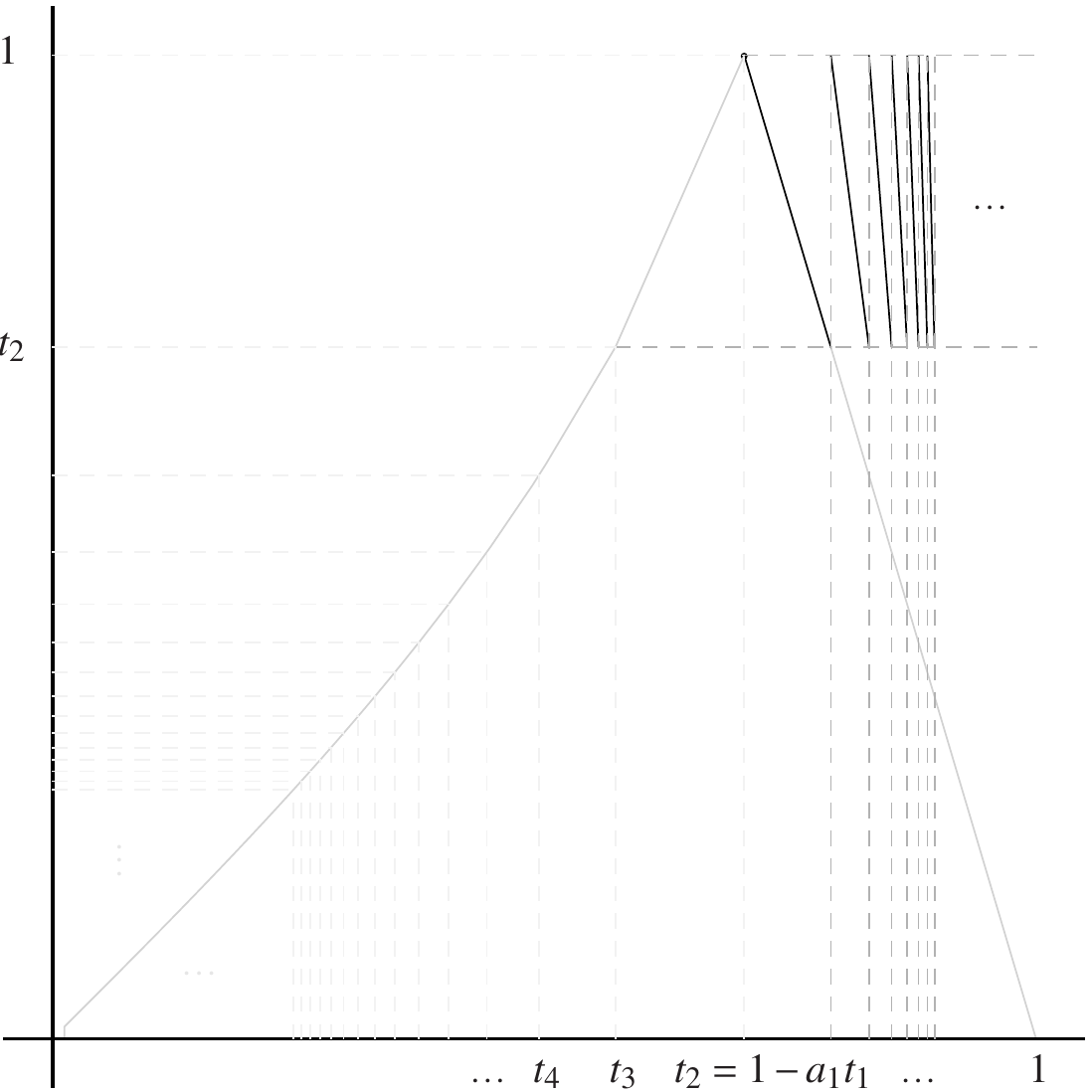}
 }}
 \hspace{0.5em}
 \subfloat[$\delta = 1$.]{
 \scalebox{0.5}{
  \includegraphics{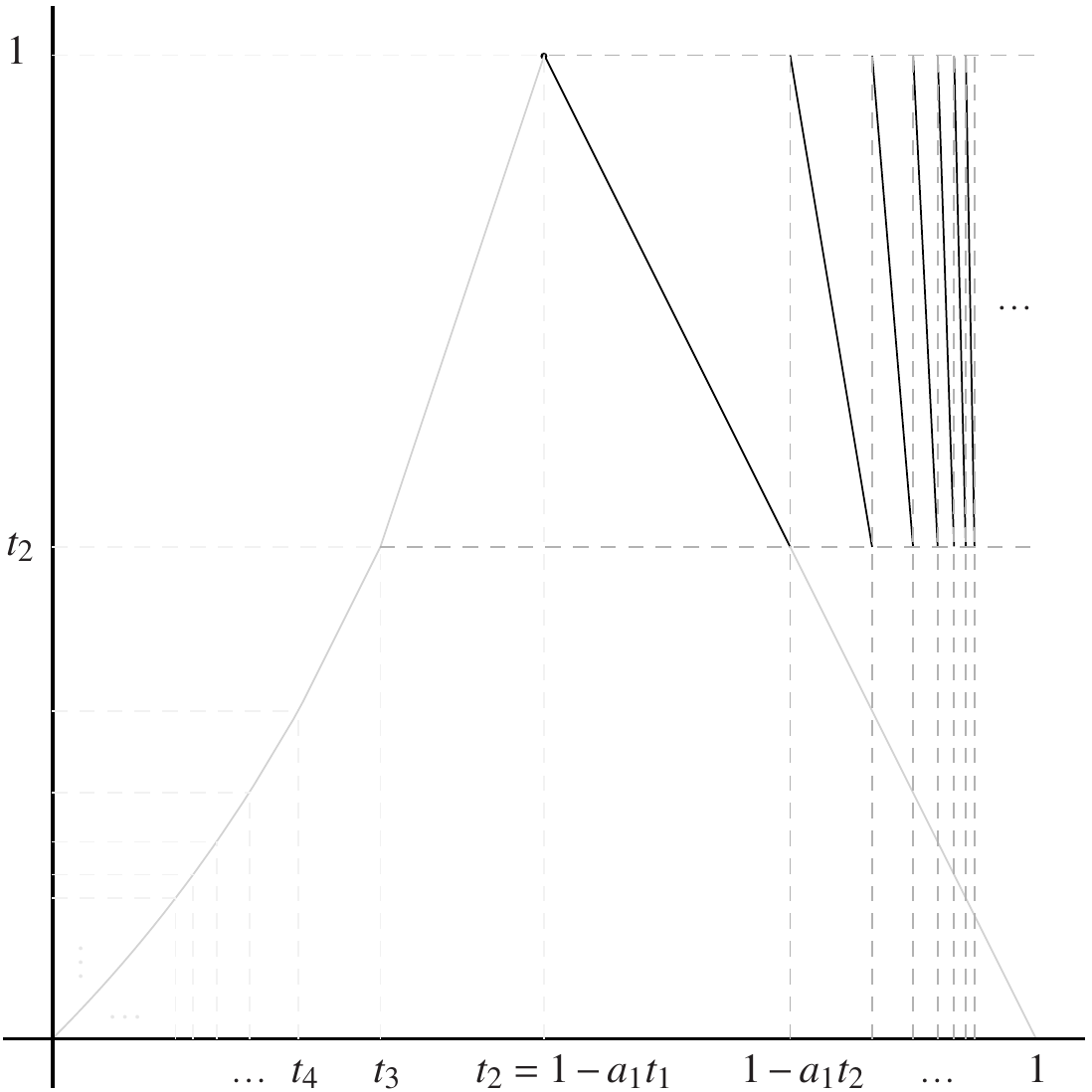}
 }}
 \end{center}
 \caption{Plot of $G_{\alpha}$, where $t_{n} = n^{-\delta}$ for all $n \in \mathbb{N}$.}
 \label{Fig2}
 \end{figure}

We denote the open unit disk in $\mathbb{C}$ by $\mathbb{D} \coloneqq \{ z \in \mathbb{C} : \lvert z \rvert < 1 \}$, its closure by $\overline{\mathbb{D}}$ and its boundary by $\mathbb{S}$.  Given $z \in \overline{\mathbb{D}}$, define $R_{n}, R(z) \colon \mathcal{L}_{\mu_{\alpha}}^{1}([0,1]) \to \mathcal{L}_{\mu_{\alpha}}^{1}([0,1])$ by 
\[
R_{n}(v) \coloneqq \widehat{F}_{\alpha}(v \cdot \mathds{1}_{\{\phi = n \}} ) = \mathds{1}_{\overline{A}_{1}} \cdot \widehat{F}_{\alpha}^{n}(v \cdot \mathds{1}_{\{\phi = n \}})
\quad \text{and} \quad
R(z) \coloneqq \sum_{n \in \mathbb{N}} z^{n} R_{n}.
\]
It is an easy exercise to show that $R(1)$ is a version of the transfer operator of the map $G_{\alpha}$. Namely, for all $v \in \mathcal{L}_{\mu_{\alpha}}^{1}([0,1])$ and all measurable functions $w$ with $\lVert w \rVert_{\infty}$ finite, we have that
\begin{equation}\label{eq:TransferOfInduced}
\int R(1)(v)\cdot w \, \mathrm{d}\mu_{\alpha}=\int v \cdot w\circ G_{\alpha} \cdot \mathds{1}_{\overline{A}_{1}} \mathrm{d}\mu_{\alpha}.
\end{equation}
We will see in Proposition~\ref{prop:conditions_H1_H2} that $(\mathcal{B}_{\alpha}, \lVert \cdot \rVert_{\mathcal{B}_{\alpha}})$ is a Banach space, that the operators $R_{n}$ and $R(1)$ map $\mathcal{B}_{\alpha}$ into itself and that the following properties are fulfilled.
\begin{description}[leftmargin=*]
\item[(H1)] There exists a constant $c > 0$ such that the operator $R_{n} \colon \mathcal{B}_{\alpha} \to \mathcal{B}_{\alpha}$ is a bounded linear operator with $\lVert R_{n} \rVert_{\mathrm{op}} \leq c \mu( \{\phi = n \} )$, for all $n \in \mathbb{N}$.  (Here, the operator norm $\lVert \cdot \rVert_{\mathrm{op}}$ is taken with respect to the Banach space $(\mathcal{B}_{\alpha}, \lVert \cdot \rVert_{\mathcal{B}_{\alpha}})$.)
\item[(H2)]
\hspace{-1.05em}
\parbox[t]{0.9\textwidth}{
\vspace{-0.9em}
\begin{enumerate}[label={(\roman*)},leftmargin=*]
\item The operator $R(1)$ restricted to $\mathcal{B}_{\alpha}$ has a simple isolated eigenvalue at $1$.
\item For each $z \in \overline{\mathbb{D}} \setminus \{ 1 \}$, the value $1$ is not in the spectrum of $R(z)\vert_{\mathcal{B}_{\alpha}}$.
\end{enumerate}
}
\end{description}
A result that will be crucial in the proof of Theorem~\ref{thm:main1} is \cite[Theorem 2.1]{MT:2011}. In order to see how this result reads in our situation, note, for a $\delta$-expansive $\alpha$-Farey map, that $\mu( \{ y \in A_{1} : \phi(y) > n \} ) = t_{n+1} \sim  l(n) n^{-\delta}$, which is essential in the proof of \cite[Theorem 2.1]{MT:2011} given in \cite{MT:2011}. Further, since $t_{n} \sim l(n) n^{-\delta}$ and $t_{n+1} < t_{n}$, for all $n \in \mathbb{N}$, Karamata's Tauberian Theorem for power series \cite[Corollary 1.7.3]{BGT:1987} implies that, for $\delta \in (0, 1)$,
\begin{equation}\label{eq:Karamata}
w_{n} \sim \overline{\Gamma}_{\delta} n^{1 - \delta} l(n).
\end{equation}
Here, $\overline{\Gamma}_{\delta} \coloneqq \Gamma(1 - \delta)/\Gamma(2-\delta)$.

\begin{theorem}[{\cite[Theorem 2.1]{MT:2011}}]\label{thm:MT2011:THM2.1}
Assuming the above setting, in particular that conditions (H1) and (H2) are satisfied, we have that
\[
\lim_{n \to +\infty} \sup_{v \in \mathcal{B}_{\alpha}: \lVert v \rVert_{\mathcal{B}_{\alpha}} = 1} \left\lVert w_{n} \mathds{1}_{\overline{A}_{1}} \cdot \widehat{F}_{\alpha}^{n}(v) -  \Gamma_{\delta} \int v \, \mathrm{d}\mu_{\alpha} \right\rVert_{\mathcal{B}_{\alpha}} = 0.
\]
\end{theorem}

In the sequel, we will also use of the following auxiliary results, where we set $\Sigma \coloneqq \{ 0, 1\}$ and for each $n \in \mathbb{N}$ and for each word $\omega \coloneqq (\omega_{1}, \omega_{2}, \dots, \omega_{n}) \in \Sigma^{n}$ we let $F_{\alpha, \omega} \colon [0, 1] \to [0, 1]$ denote the function $F_{\alpha, \omega_{1}} \circ F_{\alpha, \omega_{2}} \circ \dots \circ F_{\alpha, \omega_{n}}$.  If $\omega$ is equal to the empty word, then we set $F_{\alpha, \omega}$ to be equal to the identity map.

\begin{lemma}\label{lamma:PartIIClaim1}
Let $F_{\alpha} \colon [0, 1] \to [0, 1]$ denote an arbitrary $\alpha$-Farey map.  For each $k \in \mathbb{N}$, we have that
\[
\widehat{F}_{\alpha}^{k}(u) = \sum_{n \in \mathbb{N}} \sum_{\omega \in \Sigma^{k}} c_{n, \omega} u \circ F_{\alpha, \omega} \cdot \mathds{1}_{A_{n}},
\]
where the constants $c_{n, \omega}$ are given recursively by
\begin{equation}\label{eq:constants}
\begin{aligned}
c_{n, (0)} &\coloneqq t_{n+1}/t_{n}, 
& c_{n, (\omega_{1}, \dots, \omega_{k},0)} &\coloneqq c_{n, (0)} c_{n+1,\omega},\\
c_{n, (1)} &\coloneqq 1 - t_{n+1}/t_{n}, 
& c_{n, (\omega_{1}, \dots, \omega_{k},1)} &\coloneqq c_{n+1, (1)} c_{1,\omega}.
\end{aligned}
\end{equation}
In particular, letting $\displaystyle{0_{k} \coloneqq (\underbrace{0,0, \dots, 0}_{\text{$k$-times}})}$, we have that $c_{1, 0_{k}} = t_{k+1}$, for each $k \in \mathbb{N}$.
\end{lemma}

\begin{proof}
We proceed by induction on $k$.  The start of the induction is an immediate consequence of \eqref{eq:defofThat}.  Suppose that the statement is true for some $k \in \mathbb{N}$.  We then have that
\[
\begin{aligned}
&\widehat{F}_{\alpha}^{k+1}(u)
=\widehat{F}_{\alpha} \left(\widehat{F}_{\alpha}^{k}(u)\right)
=\widehat{F}_{\alpha} \left( \sum_{n \in \mathbb{N}} \sum_{\omega \in \Sigma^{k}} c_{n, \omega} u \circ F_{\alpha, \omega} \cdot \mathds{1}_{A_{n}} \right)\\
&=\sum_{m = 1}^{\infty} \left( \sum_{n \in \mathbb{N}} \sum_{\omega \in \Sigma^{k}} \frac{t_{m+1}}{t_{m}} c_{n, \omega} u \circ  F_{\alpha, \omega} \circ  F_{\alpha, 0} \cdot  \mathds{1}_{A_{n}} \circ  F_{\alpha, 0}\right.\\
&\left.\hspace{1cm} + \left(1 - \frac{t_{m+1}}{t_{m}}\right) c_{n, \omega} u \circ  F_{\alpha, \omega} \circ  F_{\alpha, 1} \hspace{-0.1em}\cdot \hspace{-0.1em} \mathds{1}_{A_{n}} \circ  F_{\alpha, 1} \right) \cdot  \mathds{1}_{A_{m}}\\
&=\sum_{m = 1}^{\infty} \left(\sum_{\omega \in \Sigma^{k}} \frac{t_{m+1}}{t_{m}} c_{m+1, \omega} u \circ F_{\alpha, \omega}\circ F_{\alpha, 0} + \left(1 - \frac{t_{m+1}}{t_{m}}\right) c_{1, \omega} u \circ F_{\alpha, \omega} \circ F_{\alpha, 1} \right) \cdot \mathds{1}_{A_{m}}.
\end{aligned}
\]
This completes the proof of \eqref{eq:constants}.  The remaining assertion is proven by a straight forward inductive argument, using the defining relations given in \eqref{eq:constants}.
\end{proof}

\begin{lemma}\label{lem:partIIclaim3}
For each $n \in \mathbb{N}$, we have that
\begin{equation}\label{eq:Thatonpetals1}
\widehat{F}_{\alpha}^{n-1}(\mathds{1}_{A_{n}}) = t_{n} \mathds{1}_{A_{1}}
\end{equation}
and hence, by the definition of the norm, $\lVert \widehat{F}_{\alpha}^{n-1}(\mathds{1}_{A_{n}}) \rVert_{\mathcal{B}_{\alpha}} = \lVert \widehat{F}_{\alpha}^{n-1}(\mathds{1}_{A_{n}}) \rVert_{\infty} = t_{n}$.
\end{lemma}

\begin{proof}
For $n = 1$ the result is immediate.  For $n \neq 1$, we have, by Lemma~\ref{lamma:PartIIClaim1}, that, on $[0, 1)$,
\[
\widehat{F}_{\alpha}^{n-1}(\mathds{1}_{A_{n}})
= \sum_{k = 1}^{\infty} c_{k, 0_{n-1}} \mathds{1}_{A_{n}} \circ F_{\alpha,0_{n-1}} \cdot \mathds{1}_{A_{k}}
= \sum_{k = 1}^{\infty} c_{k, 0_{n-1}} \mathds{1}_{A_{1}} \cdot \mathds{1}_{A_{k}}
= t_{n} \mathds{1}_{A_{1}}.
\]
To complete the proof, we need to evaluate the function $\widehat{F}_{\alpha}^{n-1}(\mathds{1}_{A_{n}})$ at the point $1$ for $n \geq 2$.  By Lemma~\ref{lamma:PartIIClaim1}, we have that
\[
\widehat{F}_{\alpha}^{n-1}(\mathds{1}_{A_{n}})(1)
= \sum_{n \in \mathbb{N}} \sum_{\omega \in \Sigma} c_{n, \omega} \mathds{1}_{A_{n}} \circ F_{\alpha, \omega}(1) \cdot \mathds{1}_{A_{n}}(1)
= \sum_{\omega \in \Sigma} c_{1, \omega} \mathds{1}_{A_{n}} \circ F_{\alpha, \omega}(1)
= 0.
\]
This completes the proof.
\end{proof}

We will now show that conditions (H1) and (H2) are satisfied for every $\delta$-expansive $\alpha$-Farey system and for the Banach space $(\mathcal{B}_{\alpha}, \lVert \cdot \rVert_{\mathcal{B}_{\alpha}})$.

\begin{proposition}\label{prop:conditions_H1_H2}
The pair $(\mathcal{B}_{\alpha}, \lVert \cdot \rVert_{\mathcal{B}_{\alpha}})$ forms a Banach space and for a $\delta$-expansive $\alpha$-Farey system, the operators $R_{n}$ and $R(1)$ map $\mathcal{B}_{\alpha}$ into itself.  Moreover, (H1) and (H2) are satisfied.
\end{proposition}

In the proof of the above proposition we will make use of the following lemma.

\begin{lemma}\label{lemma:c_10}
For any $\alpha$-Farey map $F_{\alpha}$, we have that $c_{1, 10_{n-1}} = \mu_{\alpha}(\{ \phi = n \}) = a_{n} = t_{n} - t_{n+1}$, where $\displaystyle{10_{n} \coloneqq (1,\underbrace{0,0, \dots, 0}_{\text{$n$-times}})}$, for each $n \in \mathbb{N}$.
\end{lemma}

\begin{proof}
By construction of the $\alpha$-Farey map $F_{\alpha}$, we have that $\{ \phi = 1 \} = [1 - a_{1}t_{1}, 1 - a_{1}t_{2}]$ and that $\{ \phi = n \} = (1- a_{1} t_{n}, 1- a_{1} t_{n+1}]$, for all integers $n > 1$.  Thus,
\begin{equation}\label{eq:nu_phi_n}
\mu_{\alpha}(\{ \phi = n \})
= \int \mathds{1}_{\{ \phi = n \}} \cdot \frac{\mathrm{d}\mu_{\alpha}}{\mathrm{d}\lambda} \, \mathrm{d}\lambda
= {a_{1}}^{-1} \lambda(\{ \phi = n \})
= t_{n} - t_{n+1}.
\end{equation}
We will now show by induction on $n$ that, for each $k \in \mathbb{N}$,
\begin{equation}\label{eq:constant_01}
c_{k, 10_{n-1}} = \frac{t_{k+n-1}-t_{k+n}}{t_{k}}.
\end{equation}
From \eqref{eq:defofThat}, we have that $c_{k,(1)} = 1 - t_{k+1}/t_{k} = (t_{k} - t_{k+1})/t_{k}$, for each $k \in \mathbb{N}$.  Suppose that the statement in \eqref{eq:constant_01} is true for some $n \in \mathbb{N}$.  From \eqref{eq:constants}, we have that 
\[
c_{k, 10_{n}} \coloneqq (t_{k+1}/t_{k}) c_{k+1, 10_{n-1}},
\]
for each $k \in \mathbb{N}$, which gives
\[
c_{k, 10_{n}} = \frac{t_{k+1}}{t_{k}} \frac{t_{k+n}-t_{k+n+1}}{t_{k+1}} = \frac{t_{k+n}-t_{k+n+1}}{t_{k}}.
\]
This completes the proof of the statement in \eqref{eq:constant_01}.

Setting $k = 1$ in \eqref{eq:constant_01}, we obtain that $c_{1, 10_{n-1}} = t_{n} - t_{n+1}$, for all $n \in \mathbb{N}$.  Combining this with \eqref{eq:nu_phi_n}, completes the proof.
\end{proof}

\begin{proof}[Proof of Proposition~\ref{prop:conditions_H1_H2}]
It is shown in \cite[Section 1]{AD:2001} that the pair $(\mathcal{B}_{\alpha}, \lVert \cdot \rVert_{\mathcal{B}_{\alpha}})$ forms a Banach space.

We now prove that condition \textnormal{(H1)} holds and the invariance of $\mathcal{B}_{\alpha}$.  For this, let $u \in \mathcal{B}_{\alpha}$ and fix $k \in \mathbb{N}$.  Applying Lemmas~\ref{lamma:PartIIClaim1} and \ref{lemma:c_10} we have that
\[
R_{k}(u) = \mathds{1}_{\overline{A}_{1}} \cdot \widehat{F}_{\alpha}^{k}(\mathds{1}_{\{ \phi = k \}} \cdot u) = \mathds{1}_{\overline{A}_{1}} \cdot \mu_{\alpha}(\{ \phi = k \}) \cdot u \circ F_{\alpha, 10_{k-1}}.
\]
Hence, by definition of the partition $\beta_{\alpha}$, we have that $\lVert R_{k}(u) \rVert_{\mathcal{B}_{\alpha}} \leq \mu_{\alpha}(\{ \phi = k \}) \lVert u \rVert_{\mathcal{B}_{\alpha}}$, and so, the operator $R_{k}$ maps $\mathcal{B}_{\alpha}$ into itself.  Further, by definition of $R(1)$, this gives that
\[
\lVert R(1)u \rVert_{\mathcal{B}_{\alpha}}
\leq \sum_{n \in \mathbb{N}} \lVert R_{n}(u) \rVert_{\mathcal{B}_{\alpha}}
\leq \sum_{n \in \mathbb{N}} \mu_{\alpha}(\{ \phi = n \}) \lVert u \rVert_{\mathcal{B}_{\alpha}} = \lVert u \rVert_{\mathcal{B}_{\alpha}},
\]
and so, the operator $R(1)$ maps $\mathcal{B}_{\alpha}$ into itself.  Linearity of $R_{k}$ and $R(1)$ follows from the linearity of $\widehat{F}_{\alpha}$.

For the proof of property \textnormal{(H2)(i)}, observe that $G_{\alpha}$ is a piecewise linear expansive map with the following properties.
\begin{enumerate}[label={(\roman*)},leftmargin=*]
\item  On the set $\{ \phi = n \}$, the absolute value of the derivative of $G_{\alpha}$ is equal to $1/(t_{n} - t_{n+1})$.  Moreover, since $( t_{n} )_{n \in \mathbb{N}}$ is a positive monotonically decreasing sequence which is bounded above by $1$, it follows that there exists a constant $c > 1$ with $1/(t_{n} - t_{n+1}) > c$, for all $n \in \mathbb{N}$.
\item The partition $\beta_{\alpha}$ is a countable-infinite partition of $A_{1}$ and $G_{\alpha}(\{ \phi = 1 \}) = \overline{A}_{1}$ and $G_{\alpha}(\{ \phi = n \}) = A_{1}$ if $n \geq 2$, and hence, $\mu_{\alpha}(G_{\alpha}(\{ \phi = n \})) = \mu_{\alpha}(A_{1}) = 1$, for all $n \in \mathbb{N}$.  Moreover, the $\sigma$-algebra generated by $\{ G_{\alpha}^{-n}(\{ \phi = m \}) : n, m \in \mathbb{N} \}$ is equal to the Borel $\sigma$-algebra on $\overline{A}_{1}$.
\item For each $n \in \mathbb{N}$ and $\psi \coloneqq (1, \underbrace{0, 0, \dots, 0}_{(n-1)-\text{times}}, 1 )$, we have that
\[
F_{\alpha, \psi}([0, 1]) = \overline{\{ \phi = n \}} \quad \text{and} \quad \frac{\mathrm{d} \mu_{\alpha} \circ F_{\alpha, \psi}}{\mathrm{d} \mu_{\alpha}} = t_{n} - t_{n+1} = a_{n}.
\]
\end{enumerate}
Given these properties, \textnormal{(H2)(i)} is a consequence of \cite[Theorem 1.6]{AD:2001}: the proof of which is based on the \textit{Theorem on the difference of two norms} by Ionescu-Tulcea and Marinescu \cite{ITM:1950}.

For the proof of property \textnormal{(H2)(ii)},  we distinguish between the cases $z \in \mathbb{D}$ and  $z \in \mathbb{S}^{1} \setminus \{ 1\}$.
\begin{description}[leftmargin=*]
\item[\textit{Case 1}. ($z \in \mathbb{D}$)]
Sarig showed in \cite[Section 3]{S:2002} that
\begin{equation}\label{eq:relation}
T(z) \circ R(z) u = (T(z) u) - u,
\end{equation}
where the operators $T(z), T_{n} \colon \mathcal{L}^{1}_{\mu_{\alpha}}([0, 1]) \to \mathcal{L}_{\mu_{\alpha}}^{1}([0, 1])$ are defined by
\[
T(z) \coloneqq \sum_{n \in \mathbb{N}_{0}} z^{n} \cdot T_{n} \quad \text{and} \quad T_{n}(u) \coloneqq \mathds{1}_{A_{1}} \cdot \widehat{F}_{\alpha}^{n}(\mathds{1}_{A_{1}} \cdot u).
\]
By way of contradiction, suppose that $1$ is an eigenvalue of $R(z)$ restricted to $\mathcal{B}_{\alpha}$.  Then there exists a non-zero measurable function $w \in \mathcal{B}_{\alpha}$ such that $R(z) w = w$.  Substituting this into \eqref{eq:relation} shows that $w$ is equal to zero $\mu_{\alpha}$-almost everywhere, which gives a contradiction.
\item[\textit{Case 2}. ($z \in \mathbb{S}^{1} \setminus \{ 1\}$)] We will now show that $1$ is not an eigenvalue of $R(z)$.  (This part of the proof is based on the proof of \cite[Lemma 6.7]{G:2004}.) Since $z \in \mathbb{S}^{1} \setminus \{ 1\}$, there exists a $t \in (0, 2\pi)$ such that $z = \e^{i t}$.  Suppose that $R(z) f = f$, for some non-zero $f \in \mathcal{B}_{\alpha}$.  Let $\mathcal{L}^{2}_{\mu_{\alpha}}(A_{1})$ denote the class of complex-valued measurable functions $f$ with domain $\overline{A}_{1}$ for which  $\lvert f \rvert^{2}$ is $\mu_{\alpha}$-integrable, and let it be equipped with the standard $\mathcal{L}^{2}_{\mu_{\alpha}}$-inner product,
\[
\langle u_{1}, u_{2} \rangle \coloneqq \int u_{1} \cdot \overline{u}_{2} \, \mathrm{d}\mu_{\alpha}.
\]
For each $u \in \mathcal{L}^{2}_{\mu_{\alpha}}(A_{1})$ set $\lVert u \rVert_{2} \coloneqq \langle u, u \rangle^{1/2}$.  Further, set $\mathcal{L}^{\infty}(A_{1}) \coloneqq \{ v: \overline{A}_{1} \to \mathbb{C} : \lVert v\rVert_{\infty} < +\infty \}$ and define $V \colon \mathcal{L}^{\infty} \to \mathcal{L}^{\infty}$ by $V(u) \coloneqq \e^{-i t \phi} \cdot u \circ G_{\alpha}$.  Noting that $R(z)v = R(1)(\e^{i t \phi} \cdot v)$ and using \eqref{eq:TransferOfInduced}, we have that, for all $v \in \mathcal{B}_{\alpha}$ and $u \in \mathcal{L}^{\infty}(A_{1})$,
\[
\langle u, R(z)v \rangle
= \int \overline{u} \cdot R(z)v \, \mathrm{d}\mu_{\alpha}
= \int \overline{u} \cdot R(1)(\e^{i t \phi} \cdot v) \, \mathrm{d}\mu_{\alpha}
= \int \overline{u} \circ G_{\alpha} \cdot \e^{i t \phi} \cdot v \, \mathrm{d}\mu_{\alpha}
= \langle V(u), v \rangle.
\]
Further,
\begin{equation}
\begin{aligned}
{\lVert V(f) - f \rVert_{2}}^{2} 
&= {\lVert V(f) \rVert_{2}}^{2} - 2  \mathfrak{Re} \langle V(f), f \rangle + {\lVert f \rVert_{2}}^{2}\\ 
&= {\lVert V(f) \rVert_{2}}^{2} - 2  \mathfrak{Re} \langle f, R(z)(f) \rangle + {\lVert f \rVert_{2}}^{2}\\ 
&= {\lVert V(f) \rVert_{2}}^{2} - 2  \mathfrak{Re} \langle f, f \rangle + {\lVert f \rVert_{2}}^{2}\\ 
&= {\lVert V(f) \rVert_{2}}^{2} - {\lVert f \rVert_{2}}^{2}
\end{aligned}
\label{eq:*}
\end{equation}
and, as $G_{\alpha}$ preserves the measure $\mu_{\alpha}$ restricted to $\overline{A}_{1}$, we have that
\begin{equation}\label{eq:**}
{\lVert V(f) \rVert_{2}}^{2} = \int \lvert f \rvert^{2} \circ G_{\alpha} \, \mathrm{d}\mu_{\alpha} = \int \lvert f \rvert^{2} \, \mathrm{d}\mu_{\alpha} = {\lVert f \rVert_{2}}^{2}.
\end{equation}
Combining \eqref{eq:*} and \eqref{eq:**}, it follows that $V(f) - f$ vanishes $\mu_{\alpha}$-almost everywhere on $\overline{A}_{1}$.  Thus, by taking the modulus, the ergodicity of $G_{\alpha}$ implies that $\lvert f \rvert$ is equal to a constant, \mbox{$\mu_{\alpha}$-almost} everywhere on $\overline{A}_{1}$.  As $f$ does not vanish $\mu_{\alpha}$-almost everywhere, this constant is non-zero, and so, we obtain that  $\e^{-i t \phi} = f / (f \circ G_{\alpha})$ almost everywhere on $\overline{A}_{1}$.  Now, for each $n \in \mathbb{N}$, let $I_{n} \subseteq \{\phi=n\}$ be the interval of positive measure, such that $G_{\alpha} (I_{n}) = \{\phi=n \}$ and let
\[
J_{n}\coloneqq \{ x\in I_{n}\colon G_{\alpha}(x) \notin I_{n} \; \text{and} \; \e^{it\phi(x)} = f(x)/f\circ G_{\alpha}(x)\}.
\]
Since $\e^{it\phi} = f/f\circ G_{\alpha}$ almost everywhere, and since the map $G_{\alpha}$ is linear and expanding, we have that $\mu_{\alpha} (J_{n}) > 0$. In particular, the set $J_{n}$ is non-empty.  We claim that there exists $y_{n} \in J_{n}$ such that $f(y_{n}) = f\circ G_{\alpha}(y_{n})$, for each $n\in\mathbb{N}$. By way of contradiction, suppose that $f(x)\neq f\circ G_{\alpha}(x)$, for all $x\in J_{n}$.  Since $f$ is constant almost everywhere on $\overline{A}_{1}$, we have that $f$ is constant almost everywhere on $J_{n} \cup G_{\alpha}(J_{n})$, which gives an immediate contradiction to the assumption that $f(x)\neq f\circ G_{\alpha}(x)$, for all $x\in J_{n}$. Therefore, we have that there exists $y_{n}\in \overline{A}_{1}$, such that $\phi(y_{n}) = n$, $f(y_{n}) = f\circ G_{\alpha}(y_{n})$ and $\e^{it\phi(y_{n})} = f(y_{n})/f\circ G_{\alpha}(y_{n})$, for each $n\in\mathbb{N}$. Hence, we have that $\e^{itn}=1$, for all $n\in \mathbb{N}$, contradicting the initial choice of $t$.
\end{description}
\end{proof}

Finally, in preparation for the proof of Theorem~\ref{thm:main1}, let us make note of the following well known properties of slowly varying functions.

\begin{lemma}\label{lem:powerslowly}
Let  $L\colon [a, +\infty] \to \mathbb{R}$ be a positive slowly varying function, for some $a \in \mathbb{N}_{0}$.
\begin{enumerate}[label={(\roman*)},leftmargin=*]
\item\label{SV(i)} \textup{\cite[p.\ 2]{Seneta:1976}} For a compact interval  $I \subset \mathbb{R}^{+}$ we have that 
\[
\lim_{x \to +\infty} L( p x)/L(x) = 1
\]
holds uniformly with respect to $p \in I$, and hence, for a fixed $b \in \mathbb{R}^{+}$,
\[
\lim_{x \to +\infty} L(x - b)/L(x) = 1.
\]
\item\label{SV(ii)} \textup{\cite[p.\ 18]{Seneta:1976}} For a fixed $b \in \mathbb{R}^{+}$ we have that
\[
\lim_{x \to +\infty} L(x) x^{-b} = 0
\quad \text{and} \quad
\lim_{x \to +\infty} L(x) x^{b} = +\infty.
\]
\item\label{SV(v)} \textup{\cite[p.\ 41]{Seneta:1976}} If $L$ is continuous, strictly increasing and 
\[
\lim_{x \to +\infty} L(x) = +\infty,
\]
then, for a fixed $c\in(0,1)$,
\[
\lim_{x \to +\infty} L^{-1}( c x)/L^{-1}(x) = 0.
\]
\item\label{SV(vi)} \textup{\cite[p.\ 50]{Seneta:1976}} If $M \colon [a +1, +\infty)\ \to \mathbb{R}$ is defined to be the linear interpolation of the function 
\[
n \mapsto \sum_{k = a+1}^{n} L(k) k^{-1},
\]
then $M$ is a slowly varying function and
\[
\lim_{x \to \infty} \frac{L(x)}{M(x)} = 0. 
\]
\end{enumerate}
\end{lemma}

\section{Extending convergence: Proof of Theorem~\ref{thm:main2}}\label{J_Diploma_Thesis}

\begin{proof}[Proof of Theorem~\ref{thm:main2}]
Let us first recall that, for $x\in (0,1]$ and $n\in \mathbb{N}$,
\[
\begin{aligned}
\left(\mathcal{P}_{\alpha}^{n+1}(h_{\alpha}\cdot v)\right)(x) &= \mathcal{P}_{\alpha}\left(\left(\mathcal{P}_{\alpha}^{n}(h_{\alpha}\cdot v)\right)(x)\right)\\
&= \left(\mathcal{P}_{\alpha}^{n}(h_{\alpha}\cdot v)\right) (F_{\alpha,0}(x))\cdot\left\lvert F'_{\alpha,0}(x)\right\rvert + \left(\mathcal{P}_{\alpha}^{n}(h_{\alpha}\cdot v)\right) (F_{\alpha,1}(x))\cdot\left\lvert F'_{\alpha,1}(x)\right\rvert,
\end{aligned}
\]
which gives
\begin{equation}\label{eq:Ext}
\left(\mathcal{P}^{n}_{\alpha}(h_{\alpha}\cdot v)\right)(F_{\alpha,0}(x)) = 
\frac{(\mathcal{P}_{\alpha}^{n+1}(h_{\alpha}\cdot v))(x) - (\mathcal{P}_{\alpha}^{n}(h_{\alpha}\cdot v)) (F_{\alpha,1}(x))\cdot \lvert F'_{\alpha,1}(x) \rvert }{\lvert F'_{\alpha,0}(x) \rvert}.
\end{equation}
We proceed by induction as follows. The start of the induction is given by the assumption in the theorem. For the inductive step, assume that the statement holds for $\bigcup_{i=1}^{k} A_i$, for some $k\in\mathbb{N}$. Then consider some arbitrary $y \in A_{k+1}$, and let $x$ denote the unique element in $A_k$ such that $F_{\alpha,0}(x)=y$.  Using \eqref{eq:Ext}, the fact that $\widehat{F}_{\alpha} = h_{\alpha}^{-1}\mathcal{P}_{\alpha}(h_{\alpha}\cdot v)$ and the inductive hypothesis in tandem with the assumption that $\lim w_{n}/w_{n+1} = 1$, we obtain that
\[
\begin{aligned}
w_{n} \left(\widehat{F}^{n}_{\alpha}(v)\right)(y) &= w_{n} \left(\widehat{F}^{n}_{\alpha}(v)\right)(F_{\alpha,0}(x)) 
=\frac{w_{n}(\mathcal{P}^{n}(h_{\alpha}\cdot v ))(F_{\alpha,0}(x))}{h_{\alpha}(F_{\alpha,0}(x))}\\
&= \frac{w_{n}(\mathcal{P}_{\alpha}^{n+1}(h_{\alpha}\cdot v))(x)-\lvert F'_{\alpha,1}(x)\rvert \cdot w_n (\mathcal{P}_{\alpha}^{n}(h_{\alpha}\cdot v))(F_{\alpha,1}(x))}
{h_{\alpha}(F_{\alpha,0}(x))\cdot \lvert{F'_{\alpha,0}(x)}\rvert}\\
&\sim \frac{h_{\alpha}(x)-h_{\alpha}(F_{\alpha,1}(x)) \cdot \lvert F'_{\alpha,1}(x)\rvert}{h_{\alpha}(F_{\alpha,0}(x)) \cdot \lvert F'_{\alpha,0}(x)\rvert} \Gamma_{\delta} \int v \mathrm{d}\mu_{\alpha}
=\Gamma_{\delta}\int v \mathrm{d}\mu_{\alpha},
\end{aligned}
\]
where the last equality is a consequence of the eigenequation $\mathcal{P}_{\alpha}h_{\alpha}=h_{\alpha}$.
\end{proof}

\begin{remark}
Using an analogous proof to that given above, one can obtain that the result of Theorem~\ref{thm:main2} holds for other interval maps, such as Gibbs-Markov maps, Thaler maps and Pomeau-Manneville maps.
\end{remark}

\section{Proof of Theorem~\ref{thm:main1}}\label{sec:main1}

\subsection{Asymptotics of the $\alpha$-Farey transfer operator for $\delta = 1$.}\label{section:Counterexamples2.2}

Throughout this section, we let $([0, 1], \mathscr{B}, \mu_{\alpha}, F_{\alpha})$ be a $1$-expansive $\alpha$-Farey system.   In order to prove Theorem~\ref{thm:main1}~\ref{i}, we will use the following auxiliary results (Lemmas~\ref{lem:convergnecelemma} and \ref{lem:jsigman}).  Before which we require the following notation.  Define the function $\ell: [0, +\infty) \to \mathbb{R}$ by
\begin{equation}\label{eq:ell}
\ell(x) \coloneqq 
\begin{cases}
x/2 + 1/2 & \text{if} \; x \in [0,1],\\
t_{n+1} (x -n) + w_{n}  & \text{if} \; x \in [n, n+1], \; \text{for} \; n \in \mathbb{N}.
\end{cases}
\end{equation}
Note that $\ell$ is the linear interpolation of the function $n \mapsto w_{n}$ defined on $\mathbb{N}_{0}$, where  $w_{0} \coloneqq 1/2$.  Further, for $\sigma \in \mathbb{R}^{+}$, define
\[
j_{\sigma}(x) \coloneqq x- \ell^{-1}((1 + \sigma)^{-1}\ell(x)),
\]
for all $x \geq \ell^{-1}((1 + \sigma)/2)$.

\begin{lemma}\label{lem:convergnecelemma}
For a given $\sigma \in \mathbb{R}^{+}$, we have that $j_{\sigma}(x) \sim x$.
\end{lemma}

\begin{proof}
For $\sigma \in \mathbb{R}^{+}$, we have that
\[
\lim_{x \to \infty} \frac{j_{\sigma}(x)}{x} = 1 - \lim_{x \to \infty} \frac{\ell^{-1}( \ell(x) (1 + \sigma)^{-1} )}{x}  = 1 - \lim_{x \to \infty} \frac{\ell^{-1}( \ell(x) (1 + \sigma)^{-1} )}{\ell^{-1}(\ell(x))} = 1,
\]
where the last equality follows from the fact that $\ell$ is a positive, strictly monotonically increasing function and Lemma~\ref{lem:powerslowly} \ref{SV(v)}.
\end{proof}

Here and in the sequel we will use the following notation.  For $r \in \mathbb{R}$ we let $\lfloor r \rfloor$ denote the largest integer not exceeding $r$.

\begin{lemma}\label{lem:jsigman}
Let $(\delta_{j})_{j \in \mathbb{N}}$ denote a sequence of positive real numbers such that  $\sum_{j = 1}^{\infty} \delta_{j} t_{j} < \infty$.  If the wandering rate is moderately increasing, then
\[
\lim_{n \to \infty} \sum_{j = 0}^{n} \frac{w_{n}}{w_{n - j}} \delta_{j + 1} t_{j + 1} = \sum_{j = 1}^{\infty} \delta_{j} t_{j}.
\]
\end{lemma}

\begin{proof}
Without loss of generality, assume that $\sup \{ \delta_{j} : j \in \mathbb{N} \} = 1$ and let $\sigma \in \mathbb{R}^{+}$ be fixed.  By definition of $\ell$, we have, for $n \geq \ell^{-1}((1 + \sigma)/2)$, that
\[
\begin{aligned}
&\sum_{j = 0}^{n} \frac{w_{n}}{w_{n - j}} \delta_{j + 1} t_{j + 1}\\
&\leq
\frac{\ell(n)}{\ell(n - j_{\sigma}(n))} \sum_{j=0}^{\lfloor j_{\sigma}(n) \rfloor} \delta_{j+1} t_{j+1} + 
\frac{\ell(n)}{\ell(n (\ell(n))^{-2})} \sum_{j = \lfloor j_{\sigma}(n) \rfloor + 1}^{\lfloor n - n (\ell(n))^{-2} \rfloor} \delta_{j +1} t_{j+1} +
2 \ell(n) \sum_{j= \lfloor n - n (\ell(n))^{-2} \rfloor + 1}^{n} t_{j+1}.
\end{aligned}
\]
By definition of $j_{\sigma}(n)$, we have that $\ell(n)/\ell(n - j_{\sigma}(n)) = (1 + \sigma)^{-1}$.  Further, by Lemma~\ref{lem:powerslowly} \ref{SV(vi)} and since $( t_{j} )_{j \in \mathbb{N}}$ is regularly varying sequence of order $-1$, we have that,
\[
\lim_{n \to \infty} 2 \ell(n) \sum_{j= \lfloor n - n (\ell(n))^{-2} \rfloor + 1}^{n} t_{j+1} \leq \lim_{n \to \infty} \frac{2 t_{n+1} n}{\ell(n)} = 0.
\]
If $\lfloor j_{\sigma}(n) \rfloor + 1 > \lfloor n - n (\ell(n))^{-2} \rfloor$, then this completes the proof.  Otherwise, note that, by Lemma~\ref{lem:powerslowly} \ref{SV(ii)}, we have that $\ell$ is a slowly varying function.  Also, since $\ell$ is an unbounded monotonically increasing function we have that $n - n (\ell(n))^{-2} \sim n$ and, by Lemma~\ref{lem:convergnecelemma}, we have that $j_{\sigma}(n) \sim n$.  The above three statements in tandem with the assumptions that $\sum_{j = 1}^{\infty} \delta_{j} t_{j} < \infty$ and that the wandering rate is moderately increasing, yield the following:
\[
\lim_{n \to \infty} \frac{\ell(n)}{\ell(n (\ell(n))^{-2})} \sum_{j = \lfloor j_{\sigma}(n) \rfloor + 1}^{\lfloor n - n (\ell(n))^{-2} \rfloor} \delta_{j +1} t_{j+1} = 0.
\]
This completes the proof in the case $\lfloor j_{\sigma}(n) \rfloor + 1 \leq \lfloor n - n (\ell(n))^{-2} \rfloor$.
\end{proof}

\begin{proof}[Proof of Theorem~\ref{thm:main1}~\ref{i}]
By Theorem~\ref{thm:MT2011:THM2.1} and Proposition~\ref{prop:conditions_H1_H2}, we have for each $n \in \mathbb{N}$ that there exists $\theta_{n}: (0,1) \to \mathbb{R}$ such that $\sup \{\lvert \theta_{n}(x) \rvert : x \in \overline{A}_{1} \} = \mathfrak{o}(1/w_{n})$ and
\begin{equation}\label{eq:MT2011:delta:one}
\mathds{1}_{\overline{A}_{1}} \cdot \widehat{F}_{\alpha}^{n}(\mathds{1}_{\overline{A}_{1}} \cdot v)
= \frac{1}{w_{n}} \int \mathds{1}_{\overline{A}_{1}} \cdot v\, \mathrm{d}\mu_{\alpha} \cdot \mathds{1}_{\overline{A}_{1}} + \theta_{n} \cdot v\cdot \mathds{1}_{\overline{A}_1}.
\end{equation}
Set $\tau_{j, n} \coloneqq w_{n}/w_{n-j} -1$, for $n \in \mathbb{N}$ and $j \in \{ 0, 1, 2, \dots, n \}$.  By \eqref{eq:MT2011:delta:one}, we have on $\overline{A}_{1}$ that
\begin{equation}\label{eq:MT-inequalities}
\begin{aligned}
&\left\lvert w_{n} \widehat{F}_{\alpha}^{n}(v) - \int v \, \mathrm{d}\mu_{\alpha} \right\rvert\\
&= \left\lvert w_{n} \sum_{j = 0}^{n} \mathds{1}_{A_{1}} \cdot \widehat{F}_{\alpha}^{n - j}\left(\mathds{1}_{A_{1}} \cdot \widehat{F}_{\alpha}^{j}(v \cdot \mathds{1}_{A_{j+1}})\right)  - \int v \, \mathrm{d}\mu_{\alpha} \right\rvert\\
&= \left\lvert w_{n} \sum_{j = 0}^{n} \frac{1}{w_{n - j}} \int \widehat{F}_{\alpha}^{j}(v \cdot \mathds{1}_{A_{j+1}}) \, \mathrm{d}\mu_{\alpha} - \int v \, \mathrm{d}\mu_{\alpha} + w_{n} \sum_{j = 0}^{n} \theta_{n-j} \cdot \widehat{F}_{\alpha}^{j}(v \cdot \mathds{1}_{A_{j+1}}) \right\rvert\\
&\leq \sum_{j = 0}^{n} \tau_{n, j} \int \rvert v \cdot \mathds{1}_{A_{j+1}} \lvert \, \mathrm{d}\mu_{\alpha} + w_{n} \sum_{j = 0}^{n} \lVert \theta_{n-j} \rVert_{\infty} \lVert \widehat{F}_{\alpha}^{j}(v \cdot \mathds{1}_{A_{j+1}})\rVert_{\infty} + \sum_{j = n+1}^{\infty} \int \lvert v \cdot \mathds{1}_{A_{j+1}} \rvert \, \mathrm{d}\mu_{\alpha}.
\end{aligned}
\end{equation}
Since $v \in \mathcal{A}_{\alpha} \subseteq \mathcal{L}^{1}_{\mu_{\alpha}}([0,1])$, it follows that in the final line of \eqref{eq:MT-inequalities} the third term converges to zero.  To see that the first and the second term in the final line of \eqref{eq:MT-inequalities} converge to to zero, observe that
\begin{enumerate}[label={(\roman*)},leftmargin=*]
\item since $v \in \mathcal{A}_{\alpha}$, we have that $v \in \mathcal{L}_{\mu_{\alpha}}^{1}([0, 1])$ and, moreover,
\[
\int \lvert v \cdot \mathds{1}_{A_{j}} \rvert \, \mathrm{d}\mu_{\alpha} = \frac{t_{j}}{a_{j}} \int v \cdot \mathds{1}_{A_{j}} \, \mathrm{d}\lambda;
\]
\item since $v \in \mathcal{A}_{\alpha}$ we have that $\lVert v \rVert_{\infty}$ is finite, and so the sequence 
\[
\left( \frac{1}{a_{j}} \int v \cdot \mathds{1}_{A_{j}} \, \mathrm{d}\lambda \right)_{j \in \mathbb{N}}
\]
is a bounded sequence;
\item using Lemma~\ref{lem:partIIclaim3} together with the fact that $\widehat{F}_{\alpha}$ is positive and linear and the fact that if $v \in \mathcal{A}_{\alpha}$, then $\lVert v \rVert_{\infty}$ is finite, we have that $\lvert \widehat{F}_{\alpha}^{j-1}(v \cdot \mathds{1}_{A_{j}}) (x) \rvert \leq \lVert v \rVert_{\infty} t_{j}$;
\item given $\epsilon > 0$, there exists $N_{\epsilon} \in \mathbb{N}$ such that $\lVert \theta_{m} \rVert_{\infty} \leq \epsilon/\ell(m)$, for all $m \geq N_{\epsilon}$.
\end{enumerate}
Combining these observations with Lemma~\ref{lem:jsigman} and \eqref{eq:main1}, we have that the first and the second term in the final line of \eqref{eq:MT-inequalities} converge to to zero.  Since the arguments given above are independent of a given point in $\overline{A}_{1}$, an application of Theorem~\ref{thm:main2} now finishes the proof.
\end{proof}

\begin{remark}
In the proof of Theorem~\ref{thm:main1}~\ref{i} we have not used the specific structure of $\mathcal{B}_{\alpha}$. We only used that $\mathcal{B}_{\alpha}$ is a Banach space which satisfies conditions (H1) and (H2).  Thus, we may replace $\mathcal{B}_{\alpha}$ by an arbitrary Banach space which satisfies conditions (H1) and (H2).  For such alternative Banach spaces see Remark~\ref{rmk:rmk1}.  In doing such a substitution one may change the uniform convergence to almost everywhere uniform convergence.
\end{remark}

To conclude, we give examples of $1$-expansive $\alpha$-Farey systems and of observables which belong to the set $\mathcal{A}_{\alpha}$ and which satisfy the summability condition given in \eqref{eq:main1}.

\begin{example}\label{ex:ex1}
Let $([0, 1], \mathscr{B}, \mu_{\alpha}, F_{\alpha})$ denote a $1$-expansive $\alpha$-Farey system with moderately increasing wandering rate.  Set $u\coloneqq {f}/{h_{\alpha}}$, where $f\in\mathscr{D}_{\mu_{\alpha}}$, for
\[
\mathscr{D}_{\mu_{\alpha}} \coloneqq \{ f \colon f \in \mathcal{L}^{1}_{\mu_{\alpha}}([0,1]) \; \text{and} \; f \in C^{2}((0,1)) \; \text{with} \; f' > 0 \; \text{and} \; f''\leq 0\}.
\]
We claim that $u \in \mathcal{A}_{\alpha}$ and moreover, that $u$ satisfies the summability condition given in \eqref{eq:main1}.

We first verify that $u \in \mathcal{A}_{\alpha}$.  For this, we are required to show that $u \in \mathcal{L}^1_{\mu_{\alpha}}([0, 1])$, that $\lVert u \rVert_{\infty} < +\infty$ and that $\widehat{F}^{j-1}_{\alpha}(u \cdot \mathds{1}_{A_{j}})\in \mathcal{B}_{\alpha}$, for all $j \in \mathbb{N}$.  By definition, any function belonging to $\mathscr{D}_{\mu_{\alpha}}$ is convex and continuous on $(0, 1)$, twice differentiable and $\mu_{\alpha}$-integrable.  Thus, $u \in \mathcal{L}^{1}_{\mu_{\alpha}}([0,1])$ and  $\lVert u \rVert_{\infty} < +\infty$.  Combining this with the fact that $1/h_{\alpha}$ is $\mu_{\alpha}$ integrable, non-negative and bounded, we have that $u \in \mathcal{L}^1_{\mu_{\alpha}}([0, 1])$ and $\lVert u \rVert_{\infty} < +\infty$.  Let us now turn to the second assertion, namely that $\widehat{F}^{n-1}_{\alpha}(u \cdot \mathds{1}_{A_{n}})\in \mathcal{B}_{\alpha}$, for all $n \in \mathbb{N}$.  We immediately have that $\widehat{F}_{\alpha}^{0}( u \cdot \mathds{1}_{A_{1}}) = u \cdot \mathds{1}_{A_{1}} \in \mathcal{B}_{\alpha}$.  For $n \geq 2$, note that, if $g$ is a differentiable Lipschitz function on $\overline{A}_{1}$, then $D_{\alpha}(g) = \sup \{ \lvert g' \rvert : x\in \overline{A}_{1}\}$.  Thus, by Lemma~\ref{lem:partIIclaim3} and the chain rule, we have that, for each integer $n \geq 2$,
\begin{equation}\label{eq:ex1}
\begin{aligned}
\lVert \widehat{F}_{\alpha}^{n-1} ( u \cdot \mathds{1}_{A_{n}}) \rVert_{\mathcal{B}_{\alpha}}
&=\lVert c_{1,0_{n-1}} ((f/h_{\alpha}) \circ F_{\alpha,0_{n-1}} )\rVert_{\mathcal{B}_{\alpha}}\\
&= \left\lVert a_{n} f\circ F_{\alpha,0_{n-1}}\right\rVert_{\mathcal{B}_{\alpha}}\\
&=\left\lVert a_{n} f\circ F_{\alpha,0_{n-1}}\right\rVert_{\infty}+D_{\alpha}(a_{n} f\circ F_{\alpha,0_{n-1}})\\
&\leq a_{n}(\lVert f\rVert_{\infty}+\lVert f' \cdot \mathds{1}_{A_{n}}\rVert_{\infty}).
\end{aligned}
\end{equation}
Since $f\in \mathscr{D}_{\mu_{\alpha}}$, we have that $\lVert f \rVert_{\infty} < +\infty$ and that $0\leq f'(x)\leq (f(t_{n+1})-f(t_{n+2}))/(a_{n+1})$, for all $x\in A_{n}$. Therefore, since $a_{n} = l(n) n^{-2}$, it follows that there exists $c > 0$, such that
\[
\sum_{n \in \mathbb{N}, \, n \neq 1} a_{n} \lVert f' \cdot \mathds{1}_{A_{n}} \rVert_{\infty}
\leq \sum_{n \in \mathbb{N}, \, n \neq 1} a_{n} f'(\xi_{n+1})
= \sum_{n \in \mathbb{N}, \, n \neq 1} \frac{a_{n}}{a_{n+1}} (f(t_{n+1}) - f(t_{n+2}))
\leq c f(t_{3}).
\]
Combining this with \eqref{eq:ex1} and using the facts that the sequence $( a_{n} )_{n\in\mathbb{N}}$ is summable and that $\lVert f\rVert_{\infty}$ and $\lVert u \cdot \mathds{1}_{A_{1}} \rVert_{\mathcal{B}_{\alpha}}$ are finite, the summability condition in \eqref{eq:main1} follows. Hence it follows that $u \in \mathcal{A}_{\alpha}$.
\end{example}

\subsection{Asymptotics of the $\alpha$-Farey transfer operator for $\delta \in (1/2, 1]$}\label{MRE_alpha_delta_in_051}

\begin{proof}[Proof of Theorem~\ref{thm:main1}~\ref{ii}]
Recall that $([0, 1], \mathscr{B}, \mu_{\alpha}, F_{\alpha})$ is a $\delta$-expansive $\alpha$-Farey system and that $\Gamma_{\delta} = (\Gamma(1+\delta) \Gamma(2-\delta))^{-1}$.  From \eqref{eq:RND} and \eqref{eq:Thatonpetals1}, we have that, for all $n \in \mathbb{N}$,
\[
\begin{aligned}
\lVert \mathds{1}_{A_{1}} \cdot\widehat{F}^{n-1}_{\alpha}(v \cdot \mathds{1}_{A_{n}} /h_{\alpha}) \lVert_{\mathcal{B}_{\alpha}}
&= \lVert \mathds{1}_{A_{1}} \cdot\widehat{F}^{n-1}_{\alpha}(v \cdot \mathds{1}_{A_{n}} /h_{\alpha}) \lVert_{\infty} + D_{\alpha}(\mathds{1}_{A_{1}} \cdot\widehat{F}^{n-1}_{\alpha}(v \cdot \mathds{1}_{A_{n}}/h_{\alpha}) )\\
&= a_{n} \lVert v \cdot \mathds{1}_{A_{n}} \rVert_{\infty} + a_{n} D_{\alpha}(v \cdot \mathds{1}_{A_{n}} \circ F_{\alpha, 0_{n-1}}).
\end{aligned}
\]
Combining this with the assumptions of the theorem, there exists a constant $c' > 0$ such that $\lVert \mathds{1}_{A_{1}} \cdot\widehat{F}^{n-1}_{\alpha}(v \cdot \mathds{1}_{A_{n}} /h_{\alpha}) \lVert_{\mathcal{B}_{\alpha}} < c'$, for all $n \in \mathbb{N}$.

As in the proof of Theorem~\ref{thm:main1}~\ref{i} we have, by Theorem~\ref{thm:MT2011:THM2.1} and Proposition~\ref{prop:conditions_H1_H2}, that there exists $\theta_{n}: [0, 1] \to \mathbb{R}$ such that $\sup \{ \lvert \theta_{n}(x) \rvert : x \in \overline{A}_{1} \} = \mathfrak{o}(1/w_{n})$ and, for each $n \in \mathbb{N}_{0}$,
\[
\mathds{1}_{\overline{A}_{1}} \cdot \widehat{F}_{\alpha}^{n}(\mathds{1}_{\overline{A}_{1}} \cdot v / h_{\alpha})
= \frac{\Gamma_{\delta}}{w_{n}} \int \mathds{1}_{\overline{A}_{1}} \cdot v \, \mathrm{d}\lambda \, \mathds{1}_{\overline{A}_{1}} + \theta_{n} \cdot v / h_{\alpha} \cdot \mathds{1}_{\overline{A}_{1}}.
\]
Set $\tau_{j, n} \coloneqq w_{n} / w_{n-j} - 1$, for each $n \in \mathbb{N}$ and $j \in \{ 0, 1, 2, \dots, n\}$.  By a calculation similar as in \eqref{eq:MT-inequalities}, we have on $\overline{A}_{1}$ that
\begin{equation}\label{eq:MTeq-convergence_delta}
\begin{aligned}
\left\lvert w_{n} \widehat{F}_{\alpha}^{n}(v/h_{\alpha}) - \Gamma_{\delta} \int v \, \mathrm{d}\lambda \right\rvert
\leq& w_{n} \sum_{j = 0}^{n} \lVert \theta_{n-j} \rVert_{\infty}  \lVert \widehat{F}_{\alpha}^{j}(v \cdot \mathds{1}_{A_{j+1}}/h_{\alpha})\rVert_{\infty} \\
& + \Gamma_{\delta} \sum_{j = 0}^{n} \tau_{n, j} \int \rvert v \cdot \mathds{1}_{A_{j+1}} \lvert \mathrm{d}\lambda  + \Gamma_{\delta} \sum_{j = n + 1}^{\infty} \int \lvert v \cdot \mathds{1}_{A_{j+1}} \rvert \mathrm{d}\lambda.
\end{aligned}
\end{equation}
As $v \in \mathcal{L}^{1}_{\lambda}([0,1])$, the third summand on the RHS of \eqref{eq:MTeq-convergence_delta} tends to zero.  To see that the second term on the RHS of \eqref{eq:MTeq-convergence_delta} converges to zero, recall that the sequence $(t_{n})_{n \in \mathbb{N}}$ is positive, monotonically decreasing and bounded above by one and that by assumption $a_{n} = \delta l(n) n^{-(1+\delta)}$.  Further, by \eqref{eq:Karamata}, given $\xi > 0$, there exist constants $c'' \in \mathbb{R}$ and $N = N(\xi) \in \mathbb{N}$ such that, for all $n, m \in \mathbb{N}$ with $n \geq N$,
\[
w_{m}
\geq \overline{\Gamma}_{\delta} (\e^{-\xi}l(m) m^{1-\delta} + c'')
\quad \text{and} \quad
\overline{\Gamma}_{\delta} \e^{\xi} l(n) n^{1-\delta}
\geq w_{n}
\geq \overline{\Gamma}_{\delta} \e^{-\xi} l(n) n^{1-\delta}.
\]
This implies that given $\epsilon \in (0, 1)$, we have for all $n \in \mathbb{N}$ such that $\lceil \epsilon n \rceil > N$,
\[
\begin{aligned}
&\sum_{j = n - \lceil \epsilon n \rceil + 1}^{n} \frac{w_{n}}{w_{n-j}} \frac{l(j)}{j^{1+\delta-\eta}}\\
&\leq \frac{\overline{\Gamma}_{\delta} \e^{\xi} (l(n))^{2}}{w_{0} n^{2\delta - \eta}} + \hspace{-1mm}
\sum_{j = n - \lceil \epsilon n \rceil + 1}^{n - N} \frac{\e^{2\xi} n^{1-\delta} l(n)}{(n-j)^{1-\delta}l(n-j)}\frac{l(j)}{j^{1+\delta-\eta}}
+ \hspace{-1mm} \sum_{j = n - N + 1}^{n-1}\frac{\e^{2\xi} l(n) n^{1-\delta}}{(n-j)^{1-\delta}l(n-j) + c'' \e^{\xi}} \frac{l(j)}{j^{1+\delta-\eta}}.\\
\end{aligned}
\]
(Recall that $\eta$  is the value given in condition (b) of Theorem~\ref{thm:main1}~\ref{ii}).  A simple calculation shows that the RHS of the latter inequality converges to zero as $n$ tends to infinity.  Further, since all of the summand are positive, we have that the limit as $n$ tends to infinity exists and equals zero.  This together with Lemma~\ref{lem:powerslowly}~\ref{SV(i)} implies that, for each $\epsilon, \xi \in (0, 1)$,
\begin{equation}
\begin{aligned}
\limsup_{n \to +\infty} \sum_{j = 0}^{n} \frac{w_{n}}{w_{n-j}} \int \lvert v \cdot \mathds{1}_{A_{j+1}} \rvert \mathrm{d}\lambda
&\leq \limsup_{n \to +\infty} \sum_{j = 0}^{n -  \lceil \epsilon n \rceil} \frac{\e^{2\xi} l(n)  n^{1-\delta}}{l(n-j)  (n-j)^{1-\delta}} \int \lvert v \cdot \mathds{1}_{A_{j+1}} \rvert \mathrm{d}\lambda\\
&= \limsup_{n \to +\infty} \sum_{j = 0}^{n -  \lceil \epsilon n \rceil} \frac{\e^{2\xi} n^{1-\delta}}{(n-j)^{1-\delta}} \int \lvert v \cdot \mathds{1}_{A_{j+1}} \rvert \mathrm{d}\lambda\\
&\leq \epsilon^{\delta -1} \e^{2\xi} \int \lvert v \rvert \mathrm{d}\lambda.
\end{aligned}
\label{eq:s_bound}
\end{equation}
Furthermore, 
\[
\begin{aligned}
\liminf_{n \to +\infty} \sum_{j = 0}^{n} \frac{w_{n}}{w_{n-j}} \int \lvert v \cdot \mathds{1}_{A_{j+1}} \rvert \mathrm{d}\lambda
&\geq \liminf_{n \to +\infty} \sum_{j = 0}^{n -  \lceil \epsilon n \rceil}\frac{\e^{-2\xi} l(n)  n^{1-\delta}}{l(n-j)  (n-j)^{1-\delta}} \int \lvert v \cdot \mathds{1}_{A_{j+1}} \rvert \mathrm{d}\lambda\\
&\geq \epsilon^{\delta -1} \e^{-2\xi}  \int \lvert v \rvert \mathrm{d}\lambda.
\end{aligned}
\]
Since $\epsilon, \xi \in (0,1)$ are arbitrary, it follows that
\[
\lim_{n \to +\infty} \sum_{j = 1}^{n} \frac{w_{n}}{w_{n-j}} \int \lvert v \cdot \mathds{1}_{A_{j}} \rvert \mathrm{d}\lambda =  \int \lvert v \rvert \mathrm{d}\lambda,
\]
and hence, the second term on the RHS of \eqref{eq:MTeq-convergence_delta} converges to zero as $n$ tends to infinity. We now show that the first term on the RHS of \eqref{eq:MTeq-convergence_delta} converges to zero.  By \eqref{eq:Karamata} and the fact that $\sup \{ \lvert \theta_{n}(x) \rvert : x \in \overline{A}_{1} \} = \mathfrak{o}(1/w_{n})$, given $\xi > 0$ there exists $M = M(\xi) \in \mathbb{N}$ such that, for all $m \geq M$,
\[
\overline{\Gamma}_{\delta} \e^{-\xi} l(m) m^{1-\delta}
\leq w_{m}
\leq \overline{\Gamma}_{\delta} \e^{\xi} l(m) m^{1-\delta}
\quad \text{and} \quad
\lVert \theta_{m} \rVert_{\infty} \leq \xi/ w_{m}.
\]
Moreover, there exists a constants $c_{1}, c_{2} \in \mathbb{R}$ such that, for all $n \in \mathbb{N}_{0}$,
\[
w_{n} \geq \overline{\Gamma}_{\delta}(\e^{-\xi} l(n) n^{1-\delta} + c_{1}) \quad \text{and} \quad \rVert \theta_{n} \lVert_{\infty} \leq c_{2}/w_{n}.
\]
Furthermore, since $F_{\alpha}$ is $\delta$-expansive, by condition (b) in Theorem~\ref{thm:main1}~\ref{ii}, we have that the sequence $( a_{n} \lVert v \cdot \mathds{1}_{A_{n}} \rVert_{\infty} )_{n \in \mathbb{N}}$ is summable.  These properties together with Lemma~\ref{lem:partIIclaim3} and an argument similar to that presented in \eqref{eq:s_bound}, imply the existence of a constant $c_{3} > 0$ such that
\[
\begin{aligned}
0 &\leq\limsup_{n \to +\infty} w_{n} \sum_{j = 0}^{n} \lVert \theta_{n-j} \rVert_{\infty} \lVert \widehat{F}_{\alpha}^{j}(v \cdot \mathds{1}_{A_{j+1}}/h_{\alpha})\rVert_{\infty}\\
&\leq \limsup_{n \to +\infty} \xi \e^{2\xi} \sum_{j = 0}^{n - M}  \frac{l(n)  n^{1-\delta}}{l(n - j)  (n - j)^{1-\delta}}\lVert v \cdot \mathds{1}_{A_{j+1}} \rVert_{\infty}  a_{j + 1}
+ \limsup_{n \to +\infty}\frac{c_{2} w_{n} a_{n+1}}{w_{0}}\\
&\phantom{=} \;+ \limsup_{n \to +\infty} c_{3} \e^{2\xi} \sum_{j = n - M+1}^{n - 1}  \frac{l(n)  n^{1-\delta}}{(n-j)^{1-\delta}l(n - j) + c_{1}\e^{\xi}} a_{j + 1} \lVert v \cdot \mathds{1}_{A_{j+1}} \rVert_{\infty}\\
&= \xi \e^{2\xi} \sum_{j = 0}^{\infty} a_{j+1} \lVert v \cdot \mathds{1}_{A_{j+1}} \rVert_{\infty}.
\end{aligned}
\]
Since $\xi > 0$ was chosen arbitrarily, the result follows and, since the arguments given above are independent of a given point in $\overline{A}_{ 1}$, an application of Theorem~\ref{thm:main2} finishes the proof.
\end{proof}

\subsection{Proof of Theorem~\ref{thm:main1}~\ref{iii} - Counterexamples for $\delta \in (1/2,1)$}\label{section:Counterexamples}

In this section we provide a constructive proof of Theorem~\ref{thm:main1}~\ref{iii}.  The proof is divided into several parts.  First, we define a class of observables $\mathcal{V}$.   Second, in Proposition~\ref{prop:L1andLinfty} we will show that if $v \in \mathcal{V}$, then $v$ is bounded, of bounded variation, Riemann integrable and belongs to $\mathcal{L}_{\mu_{\alpha}}^{1}([0, 1])$.  Third, in Proposition~\ref{prop:Balpha} we will show that if $v \in \mathcal{V}$, then it belongs to the space $\mathcal{A}_{\alpha}$, and in Proposition~\ref{prop:summand} we will show that the summability condition given in \eqref{eq:main1} is satisfied for all $v \in \mathcal{V}$.  Finally, in Proposition~\ref{prop:liminfandlimsup} we will show that, if $v \in \mathcal{V}$, then 
\[
\liminf_{n \to +\infty} w_{n}  \widehat{F}_{\alpha}^{n}(v)(x) =  \Gamma_{\delta} \int v \, \mathrm{d} \mu_{\alpha}
\quad \text{and} \quad
\limsup_{n \to +\infty} w_{n}  \widehat{F}_{\alpha}^{n}(v)(x) = +\infty.
\]
Combing these results will then yield a proof of Theorem~\ref{thm:main1}~\ref{iii}

Let us now begin by defining the set $\mathcal{V}$.  We let $\mathcal{V}$ denote the class of observable $v: [0, 1] \to \mathbb{R}$ which are of the following form:
\begin{equation}\label{eq:observable}
v \coloneqq \sum_{k = 1}^{\infty} s_{k} \sum_{j = N_{k} - n_{k}}^{N_{k}} \mathds{1}_{A_{j}},
\end{equation}
where 
\[
N_{k} \coloneqq \lceil 2^{g_{1} k} \rceil, \quad n_{k} \coloneqq \lfloor 2^{g_{2} k} \rfloor \quad \text{and} \quad s_{k} \coloneqq 2^{-g_{3}k},
\]
and where $g_{1}$, $g_{2}$ and $g_{3}$ denote three positive constants, depending on $\delta$, such that
\begin{description}[leftmargin=*]
\item[(C1)] $g_{1} > (1 - \delta)^{-1}$,
\item[(C2)] $\delta g_{1} > g_{2}$,
\item[(C3)] there exists $\epsilon \in (0, \delta-1/2)$, such that  $g_{2}(\delta - \epsilon) >  (2\delta + 2 \epsilon - 1) g_{1} + g_{3}$.
\end{description}

\begin{example}\label{ex:exc1c2c31}
For  $\delta \in (1/2, 1)$, choose $g_1= (1+\epsilon)/(1-\delta)$ and $g_2=\delta/(1-\delta)$. Then it is clear that $g_1$ and $g_2$ satisfy the conditions (C1) and (C2). With these choices one immediately verifies that (C3) is equivalent to $g_3 < (1- \delta) - \rho \epsilon$, for $\rho \coloneqq (3 \delta +1 + 2 \epsilon)/(1-\delta)$. Hence, by choosing $\epsilon >0$ sufficiently small, it follows that the conditions (C1), (C2) and (C3) can be fulfilled simultaneously.
\end{example}

\begin{example}\label{ex:exc1c2c32}
For $\delta \in (1/2, 1)$, choose $g_{1} = (1 - \delta)^{-2}$, $g_{2} = (\delta^{2}+2\delta-1)(2\delta)^{-1}(1-\delta)^{-2}$, $g_{3} = 8^{-1}$ and $\epsilon = \delta(1-\delta)^{2}(2\delta^{2} + 12\delta -2)^{-1}$.  For each $\delta \in (1/2, 1)$, these values are all positive and satisfy conditions (C1), (C2) and (C3).
\end{example}

The main reason why we require the sequence $(s_{k})_{k\in \mathbb{N}}$ is to ensure that $v$ is of bounded variation.  Further, condition (C3) is only required in the proof of the second statement of Proposition~\ref{prop:liminfandlimsup}, specifically when Lemma~\ref{lemma:keyLemma} is used.

Before we begin with Proposition~\ref{prop:L1andLinfty}, we give two elementary lemmas which we will use in its proof.

\begin{lemma}\label{lemmaMVT}
If $s \in (0, 1)$ and $1 < b < a^{s}$, then
\[
\sum_{k=1}^{\infty}a^{(1-s)k}-(a^{k}-b^{k})^{1-s} \leq \sum_{k=1}^{\infty} (b/a^{s})^{k}< +\infty.
\]
\end{lemma}

\begin{proof}
By assumption, we have that $a/b > 1$ and so $1 - (b/a)^{k} \leq (1 - (b/a)^{k})^{1-s}$, which implies that $(a/b)^{k} - (a/b)^{k}(1-(b/a)^{k})^{1-s} \leq 1$, for each $k \in \mathbb{N}$.  Therefore,
\[
\sum_{k = 1}^{\infty} a^{(1-s)k} - (a^{k} - b^{k})^{1-s} 
= \sum_{k = 1}^{\infty} (b/a^{s})^{k} ( (a/b)^{k} - (a/b)^{k}(1 - (b/a)^{k})^{1-s} )
\leq \sum_{k=1}^{\infty} (b/a^{s})^{k}
< +\infty.%
\]
\end{proof}

\begin{lemma}\label{lem:NKnk-1}
For each $k \in \mathbb{N}$, we have that $N_{k+1} - n_{k+1} > N_{k}$.
\end{lemma}

\begin{proof}
 By (C1) and (C2), we have that $g_{1} > (1-\delta)^{-1}$ and $g_{2} - g_{1} < 0$.  This implies that, for all $k \in \mathbb{N}$,
\[
\begin{aligned}
\frac{N_{k+1}-n_{k+1}}{N_{k}}
&\geq \frac{2^{g_{1}(k+1)} - 2^{g_{2} (k+1)} - 1}{2^{g_{1} k}+1}\\
&=\frac{2^{g_{1}}(1-2^{(g_{2} - g_{1})(k+1)} - 2^{-g_{1}(k+1)})}{2^{-g_{1} k} + 1}\\
&\geq \frac{2^{g_{1}}(1-2^{2(g_{2} - g_{1})} - 2^{-2g_{1}})}{2^{-g_{1}} + 1}.
\end{aligned}
\]
Using (C1) and (C2) once more immediately verifies that the latter term is strictly greater than one.
\end{proof}

\begin{proposition}\label{prop:L1andLinfty}
An observable $v$ defined as in \eqref{eq:observable} is bounded, of bounded variation, Riemann integrable and belongs to the space $\mathcal{L}_{\mu_{\alpha}}^{1}([0, 1])$.
\end{proposition}

\begin{proof}
Clearly the observable $v$ is Riemann integrable.  Moreover, $v$ is measurable, as each of the atoms of $\alpha$ is measurable and $v$ is the sum of indicator functions of atoms of $\alpha$.  Further, the range of $v$ is equal to $\{ 0 \} \cup \{ s_{k} : k \in \mathbb{N} \}$, and thus, $\lVert v \rVert_{\infty} = s_{1}$.  By Lemma~\ref{lem:NKnk-1}, we have that $N_{k+1} - n_{k+1} > N_{k}$, and so the variation of $v$ is equal to $2 \sum_{k = 1}^{\infty} s_{k} = 1$, which is finite, as $s_{k} \coloneqq 2^{-g_{3} k}$ and as $g_{3}$ is positive. This shows that $v$ is of bounded variation.  It remains to show that $v$ is $\mu_{\alpha}$-integrable.  For this recall that $\mu_{\alpha}(A_{k})=t_{k}$, for each $k \in \mathbb{N}$.  Choose a positive constant $\eta < \min \{ \delta, g_{3}/g_{1} \}$ and recall that $t_{n} \sim l(n) n^{-\delta}$.  By Lemma~\ref{lem:powerslowly}~\ref{SV(ii)}, there exists a constant $c > 0$ such that $t_{n} \leq c l(n)n^{-\delta} \leq c n^{\eta- \delta}$, for each $n \in \mathbb{N}$.  Therefore, by Lemma~\ref{lemmaMVT} and Lemma~\ref{lem:NKnk-1}, we have that
\[
\begin{aligned}
\int \lvert v \rvert \, \mathrm{d}\mu_{\alpha}
&= \sum_{k = 1}^{\infty} s_{k} \sum_{j = N_{k} - n_{k}}^{N_{k}} t_{j}\\
&\leq \sum_{k = 1}^{\infty} 2^{-g_{3} k} \sum_{j = N_{k} - n_{k}}^{N_{k}} \frac{c}{j^{\delta - \eta}}\\
& \leq (1 - \delta + \eta)^{-1} c \sum_{k = 1}^{\infty} \left( 2^{-g_{3} k}\left(2^{g_{1}k(1 - \delta + \eta)} - \left(2^{g_{1}k} - 2^{g_{2}k}\right)^{1 - \delta + \eta}\right) + 2^{-g_{3}k} {N_{k}}^{\eta - \delta} \right)\\
& \leq (1 - \delta + \eta)^{-1} c \sum_{k = 1}^{\infty} \left( 2^{(g_{2} - \delta g_{1} + \eta g_{1} - g_{3})k} + 2^{-g_{3}k} {N_{k}}^{\eta - \delta} \right).
\end{aligned}
\]
The latter series converges, since $\eta < \min \{ \delta, g_{3}/g_{1} \}$, $g_{2} < \delta g_{1}$ and $N_{k} > 1$, for all $k \in \mathbb{N}$.
\end{proof}

Our next aim is to show that $v$ belongs to $\mathcal{A}_{\alpha}$ and satisfies the condition given in \eqref{eq:main1}.

\begin{proposition}\label{prop:Balpha}
An observable $v$ defined as in \eqref{eq:observable} belongs to $\mathcal{A}_{\alpha}$.
\end{proposition}

\begin{proof}
By Proposition~\ref{prop:L1andLinfty}, we have that $v \in \mathcal{L}^{1}_{\mu_{\alpha}}([0,1])$ and that $\lVert v \rVert_{\infty} = 1$.  Moreover, by Lemma~\ref{lem:partIIclaim3}, we have on $[0, 1]$, that, for each $j \in \mathbb{N}$,
\begin{equation}\label{eq:flaphaaj}
\widehat{F}_{\alpha}^{j-1}(v \cdot \mathds{1}_{A_{j}})(x) = \begin{cases}
t_{j} & \text{if} \; N_{k} - n_{k} \leq j \leq N_{k} \; \text{for some} \; k \in \mathbb{N} \; \text{and if} \; x \in A_{1},\\
0 & \text{otherwise.}
\end{cases}
\end{equation}
Therefore, $\widehat{F}^{j-1}_{\alpha}(v \cdot \mathds{1}_{A_{j}}) \in \mathcal{B}_{\alpha}$, for all $j \in \mathbb{N}$, and hence, it follows that $v \in \mathcal{A}_{\alpha}$.
\end{proof}

\begin{proposition}\label{prop:summand}
An observable $v$ defined as in \eqref{eq:observable} satisfies the summability condition given in \eqref{eq:main1}.
\end{proposition}

\begin{proof}
Lemma~\ref{lem:partIIclaim3} and \eqref{eq:flaphaaj} together imply that
\[
\sum_{k = 0}^{\infty} \lVert \widehat{F}_{\alpha}^{k}(v \cdot \mathds{1}_{A_{k+1}}) \rVert_{\infty}
= \sum_{k = 1}^{\infty} s_{k} \sum_{j = N_{k} - n_{k}}^{N_{k}} t_{j}
= \sum_{k = 1}^{\infty} s_{k} \sum_{j = N_{k} - n_{k}}^{N_{k}} \mu_{\alpha}(A_{j})
= \int \lvert v \rvert \, \mathrm{d}\mu_{\alpha}.
\]
The latter term is finite, since $v \in \mathcal{L}_{\mu_{\alpha}}^{1}([0, 1])$, by Proposition~\ref{prop:L1andLinfty}.
\end{proof}

In the proof of Proposition~\ref{prop:liminfandlimsup}, we will require the following auxiliary result.

\begin{lemma}\label{lemma:keyLemma}
For each $N \in \mathbb{N}$, the following sequence diverges to infinity:
\[
\left( s_{k} \sum_{j = N_{k} - n_{k}}^{N_{k}-N} \frac{{N_{k}^{1-\delta} l(N_{k})}}{(N_{k} - j)^{1-\delta} l(N_{k} - j)} \frac{l(j)}{j^{\delta}} \right)_{k \in \mathbb{N}}.
\]
\end{lemma}

\begin{proof}
The result follows immediately from combining the following three observations.
\begin{enumerate}[label={(\roman*)},leftmargin=*]
\item Using the facts that $\delta \in (1/2,1)$ and $\epsilon > 0$, that the sequence $(N_{k})_{k \in \mathbb{N}}$ is not bounded above and is strictly monotonically increasing, that $s_{k} \coloneqq 2^{-g_{3}k}$ and that $N$ is a fixed natural number, we have that 
\[
\lim_{k \to +\infty} s_{k} \left(\frac{N_{k}}{N_{k} - N} \right)^{\delta} {N_{k}}^{1 - 2 \delta - 2\epsilon} N^{\delta - \epsilon} = 0.
\]
\item For each $k \in \mathbb{N}$, we have that
\[
\begin{aligned}
s_{k} N_{k}^{1-2\delta - 2\epsilon} n_{k}^{\delta-\epsilon}
&\geq 2^{-g_{3} k} 2^{g_{1}(1- 2\delta - 2\epsilon)k} (2^{g_{2}(\delta - \epsilon)} -1)\\
&= 2^{(g_{1}(1-2\delta-2\epsilon) + g_{2}(\delta - \epsilon) - g_{3})k} - 2^{(g_{1}(1 - 2\delta - 2\epsilon) - g_{3}) k}.
\end{aligned}
\]
Using condition (C3) with the facts that $\delta \in (1/2,1)$, $\epsilon > 0$ and that $g_{1}, g_{2}$ and $g_{3}$ are positive, it follows that
\[
\lim_{k \in \mathbb{N}} s_{k} {N_{k}}^{1-2\delta - 2\epsilon} {n_{k}}^{\delta - \epsilon} = +\infty.
\]
\item There exist constants $\kappa, \xi > 0$ such that, for all $k \in \mathbb{N}$ sufficiently large,
\[
\begin{aligned}
&\sum_{j = N_{k} - n_{k} }^{N_{k}-N} \frac{{N_{k}}^{1-\delta} l(N_{k})}{(N_{k} - j)^{1-\delta}l(N_{k} - j)} \frac{l(j)}{j^{\delta}}\\
&\geq \left(\frac{1}{N_{k} - N}\right)^{\delta} l(N_{k} - N) l(N_{k}) N_{k}^{1-\delta} \e^{-\xi} \sum_{j = N_{k} - n_{k}}^{N_{k}-N} \frac{1}{(N_{k} - j)^{1-\delta}l(N_{k} - j)}\\
&\geq \kappa \e^{-\xi} \left(\frac{N_{k}}{N_{k} - N}\right)^{\delta}N_{k}^{1-2\delta - 2\epsilon} \left(n_{k}^{\delta - \epsilon}-N^{\delta - \epsilon}\right).
\end{aligned}
\]
Here, the first inequality follows from the facts that $l$ is a slowly varying function and that $\lim_{n \to \infty} (N_{k} - n_{k})/N_{k} = 1$ together with Lemma~\ref{lem:powerslowly}~\ref{SV(i)}.  The second inequality follows from Lemma~\ref{lem:powerslowly}~\ref{SV(ii)}, which guarantees  the existence of the constant $\kappa > 0$ such that $\kappa^{-1} n^{\epsilon} \geq l(n) \geq \kappa n^{-\epsilon}$, for all $n \in \mathbb{N}$.
\end{enumerate}
\end{proof}

\begin{proposition}\label{prop:liminfandlimsup}
For an observable $v$ as defined in \eqref{eq:observable}, we have that, on $\overline{A}_{1}$,
\begin{equation}\label{eq:mainthmiii}
\liminf_{n \to +\infty} w_{n}  \widehat{F}_{\alpha}^{n}(v) =  \Gamma_{\delta} \int v \, \mathrm{d} \mu_{\alpha}
\quad \text{and} \quad
\limsup_{n \to +\infty} w_{n}  \widehat{F}_{\alpha}^{n}(v) = +\infty.
\end{equation}
\end{proposition}

\begin{proof}
By Theorem~\ref{thm:MT2011:THM2.1} and Proposition~\ref{prop:conditions_H1_H2}, we have uniformly on $\overline{A}_{1}$ that
\[
\lim_{n \to +\infty} l(n) n^{1-\delta} \mathds{1}_{\overline{A}_{1}} \cdot \widehat{F}_{\alpha}^{n}(\mathds{1}_{A_{1}})
= (\Gamma_{\delta}/\overline{\Gamma}_{\delta}) \mu_{\alpha}(A_{1}) \mathds{1}_{\overline{A}_{1}}
= (\Gamma_{\delta}/\overline{\Gamma}_{\delta}) \mathds{1}_{\overline{A}_{1}}.
\]
Thus, given $\xi > 0$, there exists $N = N(\xi) \in \mathbb{N}$ such that, for all $n \geq N$ on $\overline{A}_{1}$,
\begin{equation}\label{eq:assumptionconvergence_BV}
\frac{\e^{\xi} \Gamma_{\delta} n^{\delta - 1}}{{\overline{\Gamma}_{\delta}}l(n)}
\geq \widehat{F}_{\alpha}^{n}(\mathds{1}_{A_{1}})
\geq \frac{\e^{-\xi} \Gamma_{\delta} n^{\delta - 1}}{{\overline{\Gamma}_{\delta}}l(n)}.
\end{equation}
We will first show the second statement in \eqref{eq:mainthmiii}.  For this, observe that by \eqref{eq:Karamata} it is sufficient to show that, on $A_{1}$,
\[
\limsup_{k \to +\infty} l(N_{k}) {N_{k}}^{1-\delta} \widehat{F}_{\alpha}^{N_{k}}(v)(x) = +\infty.
\]
In order to see this, let $\xi  > 0$ be fixed and let $p = p(\xi) \in \mathbb{N}$ denote the smallest integer for which $n_{p} > N$.  Since $\widehat{F}_{\alpha}$ is a positive linear operator, we have, for all $k > p$, that
\begin{equation}\label{eq:Nk-N25}
l(N_{k}) {N_{k}}^{1-\delta} \widehat{F}_{\alpha}^{N_{k}}(v) 
\geq s_{k} l(N_{k}) {N_{k}}^{1-\delta} \sum_{j = N_{k} - n_{k}}^{N_{k}-N} \widehat{F}_{\alpha}^{N_{k}}(\mathds{1}_{A_{j}}).
\end{equation}
Now, Lemma~\ref{lem:powerslowly}~\ref{SV(i)} implies that $\lim_{n \to +\infty} (n^{1-\delta}l(n))/((n+1)^{1-\delta}l(n+1)) = 1$.  As the sequence $(a_{n})_{n \in \mathbb{N}}$ is positive and since $a_{n} = \delta n^{-1-\delta} l(n)$ the value $r \coloneqq \inf \{ (n^{1-\delta}l(n))/((n+1)^{1-\delta}l(n+1)) \}$ is finite and strictly greater than zero.  Hence, by \eqref{eq:Thatonpetals1}, \eqref{eq:assumptionconvergence_BV} and \eqref{eq:Nk-N25} and the fact that $t_{n} \sim l(n) n^{-\delta}$, we have on $\overline{A}_{1}$ that, for each $k \in \mathbb{N}$ sufficiently large,
\[
\begin{aligned}
l(N_{k}) {N_{k}}^{1-\delta} \widehat{F}_{\alpha}^{N_{k}}(v)
&\geq \e^{-\xi}\frac{\Gamma_{\delta}}{{\overline{\Gamma}_{\delta}}} s_{k} \sum_{j = N_{k} - n_{k}}^{N_{k} - N} \frac{(N_{k}-j)^{1-\delta}l(N_{k}-j)}{(N_{k} -j +1)^{1-\delta}l(N_{k} -j +1)} \frac{N_{{k}}^{1-\delta} l(N_{k})}{(N_{k} - j)^{1-\delta}l(N_{k} - j)} t_{j}\\
&\geq \e^{-2\xi}\frac{\Gamma_{\delta}}{{\overline{\Gamma}_{\delta}}} r s_{k} \sum_{j = N_{k} - n_{k}}^{N_{k} - N} \frac{N_{{k}}^{1-\delta}l(N_{k})}{(N_{k} - j)^{1-\delta}l(N_{k} - j)} \frac{l(j)}{j^{\delta}}.
\end{aligned}
\]
By Lemma~\ref{lemma:keyLemma}, the latter term diverges.

All that remains to show is that the first statement of \eqref{eq:mainthmiii} holds.  For this, observe that, by positivity and linearity of $\widehat{F}$, Theorem~\ref{thm:MT2011:THM2.1},
Proposition~\ref{prop:conditions_H1_H2} and \eqref{eq:Thatonpetals1}, we have on $\overline{A}_{1}$ that, for each $k \in \mathbb{N}$, 
\[
\begin{aligned}
\Gamma_{\delta} \sum_{l=1}^{N_{k}}  \int v \cdot \mathds{1}_{A_{l}} \, \mathrm{d}\mu_{\alpha}
= \Gamma_{\delta} \sum_{m = 1}^{k} s_{m} \hspace{-0.1cm} \sum_{j = N_{m} - n_{m}}^{N_{m}} t_{j}
&= \sum_{m = 1}^{k} s_{m} \hspace{-0.1cm} \sum_{j = N_{m} - n_{m}}^{N_{m}} \hspace{-0.1cm} \liminf_{n \to +\infty} w_{n}\widehat{F}^{n-j+1}(\widehat{F}_{\alpha}^{j+1}(\mathds{1}_{A_{j}}))\\
&\leq \liminf_{n \to +\infty} w_{n} \widehat{F}^{n} \left( \sum_{m = 1}^{k} s_{m} \sum_{j = N_{m} - n_{m}}^{N_{m}} \mathds{1}_{A_{j}} \right)\\
&\leq \liminf_{n \to +\infty} w_{n} \widehat{F}^{n}(v).
\end{aligned}
\]
Since $k \in \mathbb{N}$ was arbitrary, the above inequalities imply that on $\overline{A}_{1}$,
\[
\liminf_{n \to +\infty} w_{n} \widehat{F}^{n}(v) \geq \Gamma_{\delta} \int v \, \mathrm{d}\mu_{\alpha}.
\]
Suppose that the latter inequality is strict, namely, suppose that there exists a constant $c > 0$ such that on $\overline{A}_{1}$,
\[
\liminf_{n \to +\infty} w_{n} \widehat{F}^{n}(v) \geq c > \Gamma_{\delta} \int v \, \mathrm{d}\mu_{\alpha}.
\]
This assumption together with \eqref{eq:Karamata} implies that, given $\xi > 0$, there exists $M = M(\xi) \in \mathbb{N}$ such that, for all $n \geq M$ and $x \in \overline{A}_{1}$,
\[
\widehat{F}^{n}(v)(x) \geq \e^{-\xi} {\overline{\Gamma}_{\delta}}^{-1} c n^{\delta - 1}/l(n).
\]
Thus, by Karamata's Tauberian Theorem for power series \cite[Corollary 1.7.3]{BGT:1987}, it follows that, for all $n \geq M$ and $x \in \overline{A}_{1}$,
\[
\begin{aligned}
\sum_{k = 1}^{n} \widehat{F}^{k}(v)(x)
&\geq \sum_{k = 1}^{M} \widehat{F}^{k}(v)(x) + \sum_{k = M+1}^{n} \e^{-\xi} {\overline{\Gamma}_{\delta}}^{-1} c k^{\delta - 1}/l(k)\\
&\geq \sum_{k = 1}^{M} \widehat{F}^{k}(v)(x) + \e^{-\xi} \overline{\Gamma}_{1-\delta} {\overline{\Gamma}_{\delta}}^{-1} c n^{\delta}/l(n).
\end{aligned}
\]
Hence,
\[
\liminf_{n \in +\infty} \frac{w_{n}}{n} \sum_{k = 1}^{n} \widehat{F}^{n}(v)(x) = 
\liminf_{n \to +\infty}n^{-\delta} l(n) \overline{\Gamma}_{\delta} \sum_{k = 1}^{n} \widehat{F}^{n}(v)(x)
\geq \overline{\Gamma}_{1-\delta} \delta^{-1} c
> \overline{\Gamma}_{1-\delta} \Gamma_{\delta} \int v \, \mathrm{d} \mu_{\alpha}.
\]
This is a contradiction, since by \eqref{eq:Karamata} and by combining Theorem~\ref{thm:MT2011:THM2.1} with Karamata's Tauberian Theorem for power series \cite[Corollary 1.7.3]{BGT:1987}, we have that the set $\overline{A}_{1}$ is a \mbox{Darling-Kac} set and therefore, by \cite[Proposition 3.7.5]{JA:1997}, the $\alpha$-Farey system is pointwise dual ergodic, meaning that, for $\mu_{\alpha}$-almost every $x \in [0, 1]$, we have that
\[
\lim_{n \in +\infty} \frac{w_{n}}{n} \sum_{k = 1}^{n} \widehat{F}^{n}(v)(x) = \overline{\Gamma}_{1-\delta} \Gamma_{\delta} \int v \, \mathrm{d} \mu_{\alpha}.
 \]
\end{proof}

\begin{proof}[Proof of Theorem~\ref{thm:main1}~\ref{iii}]
This follows from Propositions~\ref{prop:L1andLinfty}, \ref{prop:Balpha}, \ref{prop:summand} and \ref{prop:liminfandlimsup}.
\end{proof}

\end{document}